\documentclass{article}
\RequirePackage{etex} 

\usepackage[utf8]{inputenc}
\usepackage{graphicx}
\usepackage{amsmath}
\usepackage{amssymb}
\usepackage{enumitem}
\usepackage{color}
\usepackage{tikz}
\usepackage{autonum} 
\usepackage{comment}
\usepackage{amsthm}

\newtheorem{remark}{Remark}
\newtheorem{theorem}{Theorem}

\def\b0{{\bf 0}}
\def\g{{\bf g}}
\def\u{{\bf u}}

\def\btau{\boldsymbol{\tau}}

\def\S{\mathrm{\bf S}}

\def\U{\mathrm{\bf U}}
\def\V{\mathrm{\bf V}}
\def\Id{\mathrm{\bf Id}}

\def\div{\mathrm{div}}

\title{Non-isochoric stable granular models taking into account fluidisation by pore gas pressure}

\author{
Laurent Chupin\thanks{Universit\'e Clermont Auvergne, CNRS, LMBP, F-63000 Clermont-Ferrand, France},~
Thierry Dubois$^*$
}

\date{}

\begin{document}

\setlength{\parindent}{0pt}

\maketitle


\begin{abstract}
In this paper, we study non-isochoric models for mixtures of solid particles, at high volume concentration, and a gas. 
One of the motivations of this work concerns geophysics and more particularly the pyroclastic density currents which are precisely mixtures of pyroclast and lithic fragments and air.
They are extremely destructive phenomena, capable of devastating urbanised areas, and are known to propagate over long distances, even over almost flat topography.
Fluidisation of these dense granular flows by pore gas pressure is one response that could explain this behaviour and must therefore be taken into account in the models.
Starting from a gas-solid mixing model and invoking the compressibility of the gas, through a law of state, we rewrite the conservation of mass equation of the gas phase into an equation on the pore gas pressure whose net effect is to reduce the friction between the particles.
The momentum equation of the solid phase is completed by generic constitutive laws, specified as in Schaeffer \textit{et al} (2019, Journal of Fluid Mechanics,~\textbf{874}, 926--951) by a yield function and a dilatancy function. 
Therefore, the divergence of the velocity field, which reflects the ability of the granular flow to expand or compress, depends on the volume fraction, pressure, strain rate and inertial number. 
In addition, we require the dilatancy function to describe the rate of volume change of the granular material near an isochoric equilibrium state, \textit{i.e.} at constant volume.
This property ensures that the volume fraction, which is the solution to the conservation of mass equation, is positive and finite at all times. We also require that the non-isochoric fluidised model is linearly stable and dissipates energy (over time).  In this theoretical framework, we derive the dilatancy models corresponding to classical rheologies such as Drucker-Prager and~$\mu(I)$ (with or without expansion effects). The main result of this work is to show that it is possible to obtain non-isochoric and fluidised granular models satisfying all the properties necessary to correctly account for the physics of granular flows and being well-posed, at least linearly stable.
\end{abstract}

\section{Introduction}

Dense granular flows are present in many research areas, \textit{e.g.} chemistry, geophysics, biology, engineering and mathematics, in industrial applications, \textit{e.g.} pharmaceutical production processes, food industry, construction engineering, as well as in nature. In the latter case, they constitute a major source of potential danger to human life, buildings and infrastructure in inhabited areas, for example, in the case of landslides, which may be caused by seismic activity, soil instability or volcanic eruptions.
Landslides, which are a potential source of tsunamis, can be subaerial or submarine and are therefore characterized at first glance by a granular medium submerged in air or water.
Pyroclastic density currents (PDCs), see~\cite{Druitt} for instance, which are mixtures of solid particles (pyroclasts and lithic fragments) and air, are one of the most significant hazards of volcanic eruptions. PDCs are often characterized by the presence of a dense basal flow, which behaves like a fluid and can travel long distances from the source of the eruption, sometimes more than~100 km. Although understanding the mechanisms responsible for this particular behaviour of concentrated PDCs is one of the main scientific questions related to volcanic processes, modelling the interactions between the dense granular medium and the interstitial gas (air) remains a real challenge.

Due to the large amount of material carried out in volcanic processes, the small size of the particles (between~$10$ and~$100$ microns) in laboratory experiments and, the relatively high volume concentration, in the order of~$40$\% for moderately expanded flow to~$60$\% for dense flow, it seems preferable to consider the granular material and the fluid containing it as a continuous medium.
In their pioneering work on solid (particle) and fluid mixtures, Anderson and Jackson~\cite{anderson67} derived a system of coupled equations based on the principles of conservation of mass and momentum. A drag force reflects the solid/fluid interaction in each phase. 
In the particular case of a gas, by using among other things the fact that the viscosity and density of the gas phase are low, the system can be simplified resulting in a Darcy-type law relating the velocities of the two phases. An equation relating the density and pressure is then deduced from the mass conservation equation of the gas phase. For compressible gases, the system is closed with a state law, which allows the density to be eliminated as an unknown and an equation verified by the pressure of the interstitial gas is obtained. The resulting fluidized gas/particle mixing model is then formed by the conservation of mass and momentum equations for the solid phase, supplemented by the pore gas pressure equation. 
The latter intervenes through its gradient in the conservation of momentum equation. The net effect of the fluidisation is therefore to decrease the solid pressure.  
We thus obtain a system of coupled equations which models the granular flow while taking into account the interstitial gas through its pressure effects. However, a rheology for the solid phase still needs to be specified in order to have a complete and closed system.

A granular medium flows only if the stress exceeds a threshold, the yield stress, otherwise it does not deform and behaves like a solid.
The most advanced model and certainly the most used for simulating granular flows as a continuum and that accounts for this peculiar behaviour is the $\mu(I)$-rheology proposed by Jop \textit{et al}~\cite{jop_etal}. The underlying yield criterion of the $\mu(I)$-model is of the Drucker-Prager type, \textit{i.e.} the internal friction coefficient of the granular material is proportional to the solid pressure. 
Recall that the pore gas pressure acts in the granular momentum conservation equation by decreasing the solid pressure, so the friction between particles is reduced and the granular flow is fluidized, allowing it to travel greater distances. 
The question of a possible fluidisation effect of the pore gas pressure has been addressed through laboratory experiments of the collapse of dense granular columns fluidized by air injection from below in~\cite{roche12}. 
Indeed, it has been shown in~\cite{roche12} that fluidized columns flow over distances twice as long as non-fluidised columns.

Since the work of Jop \textit{et al}~\cite{jop_etal}, the $\mu(I)$-rheology has been at the centre of many researches both for its contribution in the production of numerical results, let us quote for example~\cite{lagree11,ionescu15,martin17}, and for theoretical questions related to the well-posedness of granular models based on this constitutive law~\cite{barker15,barker17,barker_gray_2017,goddard17,schaeffer19}.   
For dense granular flows with a volume concentration~$\phi$ close to the packing limit, on the order of~$60$\% per unit volume, the variation in~$\phi$ can reasonably be neglected.  In this case, the mass conservation equation implies that the granular flow is incompressible (isochoric), neglecting to model the packing and dilatant effects of the granular material when sheared. This approach, because of its apparent simplicity, is attractive and has been used to produce numerical simulations for granular column collapse, see for instance~\cite{lagree11,ionescu15}. Whilst successful results have been obtained for predicting the profile of the granular mass during its collapse as well as for estimating the velocity of the front, the incompressible~$\mu(I)$-model is ill-posed, in the sense of Hadamard, as demonstrated by Barker \textit{et al.} in~\cite{barker15}.  Indeed, at low and high inertial numbers, small perturbations grow at an exponential rate in the high wavenumber limit. The numerical solutions, when they do not blow up, depend on the grid with bands of high gradients appearing on the strain rate and pressure, as shown in~\cite{martin17}. Note that similar instabilities have been observed with a viscous Drucker-Prager model, see for example~\cite{martin17,chupin21}. A simple way to circumvent the ill-posedness of the incompressible~$\mu(I)$-model is to regularise the constitutive law as proposed in~\cite{barker_gray_2017} and implemented in~\cite{Gesenhues19}.

Taking into account the expansion of the granular material, when it is sheared, allows to regularise the~$\mu(I)$-rheology and to obtain a linearly stable model, as shown in~\cite{barker17,heyman_2017,schaeffer19}.
In~\cite{barker17}, Barker \textit{et al} introduced a yield function and a dilatancy function, both of which depend on the volume fraction, the (solid) pressure and the inertial number. The divergence of the velocity field is assumed to be proportional to the dilatancy function and the strain rate, which makes local variations of the volume of the granular medium possible. Also, the yield condition specifies the deviatoric stress in terms of a yield function. Note that, unlike the~$\mu(I)-\Phi(I)$ rheology which provides a state law defining the pressure in terms of the volume fraction, here the volume fraction evolves with the flow according to the conservation of mass equation and the rate of volume change is specified according to the yield function. As shown in~\cite{barker17}, if the yield and dilatancy functions satisfy three conditions, namely one equation and two inequalities, then the resulting compressible~$\mu(I)$-model is linearly well-posed.  

Without imposing any further constraints, it is possible to find yield and dilatancy functions that satisfy these stability conditions. However, they are in this case obtained by purely mathematical argumentation and do not take into account the physics of granular flows. On the other hand, Pailha and Pouliquen in~\cite{pailha09} have proposed a dilatancy model based on the critical state theory proposed by Roux and Radjaï (see~\cite{roux98}). This theory introduces a dilatation angle which reflects the need for a granular material to expand, by increasing the volume it occupies, when sheared. Unfortunately, the resulting constitutive laws, rewritten in the terms of the theory developed in~\cite{barker17}, do not satisfy the stability conditions so that the model may not be well-posed.

The objective of this work is to propose non-isochoric granular models, \textit{i.e.} with local variations of the volume, which take into account the effects of fluidisation by the pore gas pressure. The constitutive laws are written in terms of two functions as in the linear stability theory developed by Barker \textit{et al}~\cite{barker17,schaeffer19}.  We require that these models are linearly stable, that they are compatible with the dilatancy model of Roux and Radjaï~\cite{roux98}, that they dissipate energy over time, and that the volume fraction, solution of the conservation of mass equation, is for any time positive and bounded from above.
By compatibility with the Roux and Radjaï dilatancy model, we mean that at equilibrium the divergence of the velocity field must be zero and that the rate of volume change depends on the deviation of the volume fraction from the equilibrium state. It should be noted that the requirement that the energy of the system must be dissipated over time is motivated by two reasons, one numerical and one theoretical. Indeed, if the energy is dissipated, one can hope to prove that the model is well posed, \textit{i.e.} that a solution exists and that it can be unique. Moreover, in order to develop stable numerical schemes,~\textit{i.e.} with bounded solutions, it is more than desirable that the continuous model is dissipative. In this framework, we derive the dilatancy models obtained for specific choices of classical rheologies, such as Drucker-Prager and~$\mu(I)$. This approach allows us to derive non-isochoric fluidised granular models with the above mentioned properties.

The paper is organised as follows. In Section~\ref{sec:solid-gaz-model}, a solid-gas mixing model derived from the Anderson and Jackson's equations is described. The fluidisation of granular flows by compressible gases, in the case of a general state law, is studied in Section~\ref{sec:fluidisation}. The special case of a perfect gas (such as air) is also discussed.
In Section~\ref{sec:generic}, the fluidized model is completed by generic constitutive laws: the yield condition and the dilatancy law are defined by introducing functions similar to those used in~\cite{barker17}. Section~\ref{sec:main-properties} is devoted to the study of the main properties of this generic fluidised and non-isochoric granular model, namely energy dissipation, compatibility of the dilatancy law with equilibrium conditions, linear stability and volume fraction bounds. 
Finally, within this theoretical framework, dilatancy laws corresponding to classical rheologies, such as Drucker-Prager and~$\mu(I)$, are derived in Section~\ref{sec:examples}. The resulting models of fluidised and non-isochoric granular flows satisfy all the aforementioned properties.

\section{Solid-gas two-phase model}\label{sec:solid-gaz-model}
\subsection{Governing equations}
We consider a mixture of solid particles, \textit{i.e.} a granular phase of (constant) density~$\rho_s$, and a gas whose (variable) density is noted~$\rho_f$.
If~$\phi$ denotes the local volume fraction of the particles within the mixture, then the mass conservation for both constituents is written
\begin{align}
& \partial_t (\phi \rho_s) + \div ( \phi \rho_s \u ) = 0, \label{mass1} \\
& \partial_t ((1-\phi) \rho_f) + \div ( (1-\phi) \rho_f \u_f ) = 0, \label{mass2}
\end{align}
where~$\u$ and~$\u_f$ correspond respectively to the velocity of the granular phase and the gas phase.

The conservation of momentum equations for the two phases involve the forces between the two components. 
These equations are derived from Jackson's book~\cite{jackson00}.
A detailed explanation of each of the terms involved is also given by Pitman and Le in~\cite[Appendix A]{pitman05}. This system of equations is written as
\begin{align}
& \phi \rho_s \big( \partial_t \u + \u \boldsymbol{\cdot} \nabla \u \big) = \phi \rho_s \g  - \nabla p + \div \, \btau  - \phi \nabla p_f + \beta(\phi) (\u_f - \u),\label{mouv1} \\
& (1-\phi) \rho_f \big( \partial_t \u_f + \u_f \boldsymbol{\cdot} \nabla \u_f \big) = (1-\phi) \rho_f \g  - (1-\phi) \nabla p_f - \beta(\phi) (\u_f - \u),\label{mouv2}
\end{align}
where~$p$ and~$p_f$ correspond to the pressures within each of the phases, \textit{i.e.} the granular phase and the gas phase respectively.
The tensor~$\btau$ expresses the extra-stresses associated with the granular phase (whereas it is assumed that the only stresses associated with the gas phase are due to pressure). An explicit expression of~$\btau$ as a function of~$\phi$, $p$ and~$\nabla \u$ will be specified later in the section on the rheology of the medium.
Finally,~$\beta$ is a drag coefficient which depends on~$\phi$ and which, according to Andreotti {\it et al.}~\cite[p.290]{andreotti12}, can be written
\begin{equation}\label{beta}
\beta(\phi) = \frac{150 \, \eta_f \phi^2}{d^2 (1-\phi)},
\end{equation}
where~$d$ is the diameter of the grain and~$\eta_f$ the viscosity of the gas.

Following the work of Anderson \textit{et al}~\cite[p.331]{anderson95} (see also~\cite[equation (3.12)]{pailha09}) which compare the different contributions in the fluid momentum conservation, we could introduced approximations associated with the smallness of density of a typical gas. The equation~\eqref{mouv2} reduces to
\begin{align}\label{darcy}
\beta(\phi) (\u_f - \u) = - (1-\phi) \nabla p_f.
\end{align}
This equation can be seen as a Darcy law: it allows us to express the fluid velocity~$\u_f$ with respect to the solid one~$\u$, the gradient of the fluid pressure~$p_f$ and the volume solid fraction~$\phi$.
The equations~\eqref{mass1}, \eqref{mass2} and~\eqref{mouv1} now read
\begin{align}
& \partial_t \phi + \div ( \phi \, \u ) = 0, \label{syst1} \\
& \partial_t ((1-\phi) \rho_f) + \div ( (1-\phi) \rho_f \u ) = \div ( \kappa(\phi) \rho_f \nabla p_f ), \label{syst2} \\
& \phi \rho_s \big( \partial_t \u + \u \boldsymbol{\cdot} \nabla \u \big) = \phi \rho_s \g  - \nabla p + \div \, \btau  - \nabla p_f. \label{syst3}
\end{align}
where $\displaystyle \kappa(\phi) = \frac{(1-\phi)^2}{\beta(\phi)} = \frac{d^2 (1-\phi)^3}{150 \, \eta_f \phi^2}$. Note that this coefficient can be related to the permeability of the material {\it via} the well-known Carman-Kozeny relationship, see~\cite{carman37} or~\cite{carman97}, and~\cite{kozeny27}.

\subsection{Energy estimate}
One of the key points we wish to emphasise in this article is that the proposed model is energetically consistent, \textit{i.e.} it has an energy that decreases over time (in the absence of external forces, such gravity forces).
The energy estimates associated with this type of flow are generally obtained by performing the scalar product of the conservation of momentum equation~\eqref{syst3} by the velocity~$\u$, then integrating with respect to the spatial variable. Here, we deduce 
\begin{equation}\label{energy1}
    \int \phi \partial_t \Big(\frac{\rho_s |\u|^2}{2} \Big) + \int \phi \u\cdot \nabla \Big(\frac{\rho_s |\u|^2}{2} \Big) = \int \phi \rho_s \g \cdot \u - \int \nabla p \cdot \u + \int (\div\, \btau)\cdot \u - \int \nabla p_f\cdot \u.
\end{equation}
Note that throughout this document, the integrations with respect to the space variable are noted~$\int$. They are performed on a domain $\Omega\subset \mathbb R^2$ on which we will assume that there is no exchange with the outside: the velocities and normal stresses are assumed to be zero on the boundary~$\partial \Omega$ so that there will never be any boundary terms due to the various integrations by parts.

Multiplying then~\eqref{syst1} by $\frac{1}{2}\rho_s |\u|^2$ and integrating, we also obtain
\begin{equation}\label{energy2}
    \int \partial_t \phi \Big(\frac{\rho_s |\u|^2}{2} \Big) + \int \div(\phi \u) \Big(\frac{\rho_s |\u|^2}{2} \Big) = 0.
\end{equation}
The sum of the last two equality~\eqref{energy1} and~\eqref{energy2}, combined with integrations by parts, gives the following estimate:
\begin{equation}\label{energy3}
    \frac{d}{dt} \Big( \int \phi \frac{\rho_s |\u|^2}{2} \Big) = \int \phi \rho_s \g \cdot \u + \underbrace{\int p_f\, \div \, \u}_{A} - \underbrace{\int \btau : \nabla \u}_{B} + \underbrace{\int p\, \div \, \u}_{C}.
\end{equation}
%
So far, the three-equation model~\eqref{syst1}--\eqref{syst3} has six unknowns, namely $\phi$, $\u$, $\rho_f$, $p_f$, $p$ and~$\btau$. In order to close the system,~\eqref{syst1}--\eqref{syst3} must be completed by three closing relations. The terms A, B and C of the energy equation are closely related to the choice of closure laws. The term A depends on the pore gas pressure $p_f$ and thus reflects the fluidisation of the solid phase by the presence of the gas phase. We will see in the next section how an equation for~$p_f$ can be derived from equation~\eqref{syst2} by accounting for the gas compressibility through a (generic) state law.
A second closure relation is given by the constitutive law specifying the deviatoric stress tensor~$\btau$
as a function of~$\phi$, $\nabla \u$ and~$p$. The rheology must be physically consistent but must also allow the control of the term~$B$ in~\eqref{energy3}.
Finally, the last closure relation will specify the divergence of the velocity field~$\u$, which expresses the local change in volume and thus governs the expansion or compression of the flow when it is sheared. Note that the control of the terms~$A$ and~$C$ depends on $\div \, \u$. 

\section{Gas compressibility and fluidisation model}\label{sec:fluidisation}

The fluidisation phenomenon of the granular flow is mainly due to the fact that the gas trapped between the particles is compressible. In other words, its density~$\rho_f$ is not constant. Depending on the nature of the gas considered, a state law relating pressure~$p_f$ and density can be imposed, namely
\begin{equation}\label{gas}
    p_f = Q(\rho_f),
\end{equation}
where~$Q$ is a differentiable function. For instance, for an ideal gas, an affine relationship between pressure and density will be imposed (see Subsection~\ref{ssec:gaz-parfait}):
\begin{equation}\label{ideal-gas0}
Q(\rho_f) = p_{\mathrm{atm}} \Big( 1 - \frac{\rho_f}{\rho_f^0} \Big),
\end{equation}
where~$p_{\mathrm{atm}}=1.013 \times 10^{5} \, \mathrm{Pa}$ is the atmospheric pressure, and $\rho_f^0=1\, \mathrm{kg}.\mathrm{m}^{-3}$ corresponds to the density of air at atmospheric pressure. 

\subsection{Case of a general gas}\label{ssec:gaz-general}

In order to obtain an energy estimate and to compensate the term~$A$ in the global estimate~\eqref{energy3}, we take inspiration from the methods used for the study of compressible fluids, see for instance~\cite{lions98}.
For a general state law of the form~\eqref{gas}, the equation~\eqref{syst2} becomes
\begin{equation}\label{density-diffusion}
\partial_t ((1-\phi) {\rho}_f) + \div ( (1-\phi) {\rho}_f \u ) = \div ( \kappa(\phi) {\rho}_f {Q}'({\rho}_f) \nabla {\rho}_f ).
\end{equation}
First, we write the equation~\eqref{density-diffusion} in non-conservative form:
\begin{equation}\label{density-diffusion-non-conservative}
(1-\phi) \big( \partial_t {\rho}_f + \u \cdot \nabla {\rho}_f ) + {\rho}_f \div \, \u = \div \big( \kappa(\phi) {\rho}_f {Q}'({\rho}_f) \nabla {\rho}_f \big).
\end{equation}
By multiplying the last equation by ${H}'({\rho}_f)$ where ${H}:\mathbb R \rightarrow \mathbb R$ is any smooth function, we obtain
\begin{equation}\label{H-non-conservative}
(1-\phi) \big( \partial_t {H}({\rho}_f) + \u \cdot \nabla {H}({\rho}_f) ) + {\rho}_f {H}'({\rho}_f) \div \, \u = {H}'({\rho}_f) \div \big( \kappa(\phi) {\rho}_f {Q}'({\rho}_f) \nabla {\rho}_f \big),
\end{equation}
which can be rewritten in conservative form as
\begin{equation}\label{H-conservative}
\begin{aligned}
\partial_t ((1-\phi){H}({\rho}_f)) & + \div( (1-\phi){H}({\rho}_f) \u) \\
& + ({\rho}_f {H}'({\rho}_f) - {H}({\rho}_f)) \div \, \u = {H}'({\rho}_f) \div \big( \kappa(\phi) {\rho}_f {Q}'({\rho}_f) \nabla {\rho}_f \big).
\end{aligned}
\end{equation}
Noting that the term $A$ that we wish to control is $A = \int Q(\rho_f) \div \, \u$, it suffices to choose~${H}$ such that $x{H}'(x)-{H}(x) = {Q}(x)$ and to integrate with respect to the space variable, which leads to
\begin{equation}
- A = \frac{d}{dt} \Big( \int (1-\phi) {H}({\rho}_f) \Big) - \int {H}'({\rho}_f) \div \big( \kappa(\phi) {\rho}_f {Q}'({\rho}_f) \nabla {\rho}_f \big).
\end{equation}
Since $x{H}''(x)={Q}'(x)$, an integration by parts allows to rewrite the last term as
\begin{equation}\label{eq:energieA}
- A = \frac{d}{dt} \Big( \int (1-\phi)H(\rho_f) \Big) + \int \kappa(\phi) |\nabla p_f|^2.
\end{equation}

\begin{remark}
As previously, in all integrations by parts, the boundary terms are zero. In the latter case, this cancellation comes from the assumption of zero normal velocities at the boundary, and from Darcy's law~\eqref{darcy}. More precisely, on the boundary~$\partial \Omega$, we have
$$
\kappa(\phi) {\rho}_f {Q}'({\rho}_f) \nabla {\rho}_f.\mathbf{n} = -(1-\phi) {\rho}_f (\u_f-\u).\mathbf{n} = 0,$$
where $\mathbf{n}$ corresponds to the outward unit normal at the boundary.
\end{remark}

In practice, to determine the function~$H$ involved in the energy estimate from the function~$Q$ specifying the state law, we integrate the differential equation $x{H}'(x)-{H}(x) = {Q}(x)$. Indeed, we find
\begin{equation}\label{Q->H}
    H(x) = x\int_{c_1}^x \frac{Q(\zeta)}{\zeta^2}\mathrm d\zeta + c_2.
\end{equation}
The reader is referred to~\cite[p.36]{lions98} where similar calculations are conducted.
It can be noted here that the choice of integration constants~$c_1$ and~$c_2$ has no influence on the energy. Indeed, if we add a constant~$c_2$ to~$H$ then the quantity~$A$ appearing in equality~\eqref{eq:energieA} will be increased by $c_2 \frac{d}{dt}\big( \int (1-\phi) \big)$ which is zero due to the mass conservation equation~\eqref{syst1}.
In the same way, if we add a linear term~$c_1 x$ to~$H$ then the quantity~$A$ appearing in equality~\eqref{eq:energieA} will be increased by $c_1 \frac{d}{dt}\big( \int (1-\phi)\rho_f \big)$ which is zero due to~\eqref{density-diffusion} and the remark above.

\subsection{The particular case of an ideal gas}\label{ssec:gaz-parfait}

As mentioned at the beginning of this section, if the gas under consideration is air, it is reasonable to consider it as an ideal gas and to impose an affine relation between~$p_f$ and~$\rho_f$. More precisely, the relation~\eqref{ideal-gas0} can be written as (see~\cite{goren10}):
\begin{equation}\label{ideal-gas}
\rho_f = \rho_f^0 \Big( 1 + \frac{p_f}{p_{\mathrm{atm}}} \Big).
\end{equation}
It is then possible, for instance using $H=Q$ in~\eqref{H-conservative}, to derive from~\eqref{syst2} an equation describing the evolution of the pressure~$p_f$, namely
\begin{equation}\label{pressure-diffusion-nonlinear}
    \partial_t ((1-\phi) p_f) + \div ( (1-\phi) p_f \u ) + p_{\mathrm{atm}} \div \, \u = \div ( \kappa(\phi) (p_{\mathrm{atm}}+p_f) \nabla p_f ).
\end{equation}
Considering that the pore gas pressure in the granular medium is negligible compared to the atmospheric pressure, that is $p_f \ll p_{\mathrm{atm}}$, it is reasonable to approximate $p_f+p_\mathrm{atm}$ by $p_f$ in the right-hand side of the last equation. As a consequence, we obtain the following pressure "diffusion" equation
\begin{equation}\label{pressure-diffusion}
    \partial_t ((1-\phi) p_f) + \div ( (1-\phi) p_f \u ) + p_{\mathrm{atm}} \div \, \u = p_{\mathrm{atm}} \div ( \kappa(\phi) \nabla p_f ).
\end{equation}
In the remainder of this article, this equation will be used to describe the evolution of the pore gas pressure, although a more general model can be chosen (\textit{i.e.} for a general gas).

\begin{remark}
This convection-diffusion equation for the pore gas pressure is frequently used, sometimes in slightly different forms.
Thus, in~\cite[equation (7)]{goren10} or in~\cite[equation (7)]{mcnamara00} corrected in~\cite{mcnamara00erratum}, the authors use the non-conservative form of~\eqref{pressure-diffusion}, namely
\begin{equation}\label{pressure-diffusion-non-conservative}
    (1-\phi) \big( \partial_t p_f + \u \cdot \nabla p_f \big) + (p_{\mathrm{atm}} + p_f) \div \, \u = p_{\mathrm{atm}} \div ( \kappa(\phi) \nabla p_f ),
\end{equation}
and then use the approximation $p_{\mathrm{atm}} + p_f \approx p_{\mathrm{atm}}$. A similar equation is often used in the incompressible case,~\textit{i.e.} when~$\phi$ is constant and $\div \, \u = 0$. It then reduces to a "classical" convection/diffusion equation, and even to a diffusion equation if the transport term is not taken into account. This is for example the case in~\cite{montserrat12,roche12}.
\end{remark}

As in the case of a general state law, it is possible to estimate the term~$A$ in the energy equation~\eqref{energy3}. More precisely, if the state law is given by~\eqref{ideal-gas}, we obtain~$H$ by using~\eqref{Q->H}, namely
\begin{equation}\label{H(p)}
    H(\rho_f)
    = p_{\mathrm{atm}} \frac{\rho_f}{\rho_f^0} \Big[ \ln \Big(\frac{\rho_f}{\rho_f^0}\Big) - 1 \Big]
    = (p_{\mathrm{atm}} + p_f ) \Big[ \ln \Big( 1+ \frac{p_f}{p_{\mathrm{atm}}} \Big) - 1 \Big],
\end{equation}
so that we get
\begin{equation}\label{eq:energieA_air}
- A = \frac{d}{dt} \Big( \int (1-\phi) (p_{\mathrm{atm}} + p_f ) \Big[ \ln \Big( 1+ \frac{p_f}{p_{\mathrm{atm}}} \Big) - 1 \Big]\Big) + \int \kappa(\phi) |\nabla p_f|^2.
\end{equation}
\section{Generic model for rheology and dilatation}\label{sec:generic}

The rheology of granular media is complex and most classical models are known to be ill-posed when granular flow is assumed to be incompressible (isochoric). More precisely, the Drucker-Prager and $\mu(I)$ rheologies are linearly unstable in some regimes, see~\cite[Appendix C]{martin17} for the viscous Drucker-Prager model and~\cite{barker15} for the~$\mu(I)$-rheology.

In several recent works, see~\cite{barker15,barker17,schaeffer19,barker23}, the authors have shown that taking into account the dilation of the granular medium allows, in the two-dimensional case, to regularise these models and to remedy these instabilities. More precisely, it is shown that the rheology, \textit{i.e.} the expression of the additional stress~$\btau$, and the dilatancy law, \textit{i.e.} the expression of $\div \,\u$, should be defined in a concordant way.

In this section, we will use the same notations as in~\cite{schaeffer19}. We define the deviatoric strain-rate tensor by
\begin{equation}\label{def:S}
\S = \frac{\nabla \u + (\nabla \u)^\mathrm{T}}{2} - \frac{1}{2}(\div \, \u) \, \Id,
\end{equation}
which is a symmetric and trace-less tensor.
We also introduce the inertial number (see~\cite{andreotti12}):
\begin{equation}\label{def:I}
I = \frac{d\, |\S|}{\sqrt{p/\rho_s}}, 
\end{equation}
where $d$ represents the grain diameter of the granular medium under consideration, and where the matrix norm is defined by the second invariant of any symmetric tensor, namely $|\S|^2 = \frac{1}{2}\S:\S = \frac{1}{2} \sum_{i,j}S_{ij}S_{ij}$.

Note that in a recent paper by Barker~\textit{et al}~\cite{barker23}, the authors use the $\mu(J)$-rheology, instead of
$\mu(I)$, to describe fluidized granular flow, \textit{i.e.} granular material immersed in water. The dimensionless number~$J$, defined by
\begin{equation}
    J=\frac{\eta_f |\S|}{p},
\end{equation}
is used for granular flows with a low Stokes number ($\mathrm{St}=\frac{\rho_s d^2}{\eta_f}|\S|$) and is therefore well suited for granular flows in liquids.
In the present study, we focus on granular flows with a relatively high Stokes number, which is the case in air with a viscosity about $50$ times smaller than that of water, while taking into account the effect of the interstitial gas through the pressure equation~\eqref{pressure-diffusion}.

\subsection{Granular rheology}

The rheology of a granular flow can then be described in a fairly general way by a relationship of the form
\begin{equation}\label{def:tau}
    \btau = Z(\phi, I) p \frac{\S}{|\S|}.
\end{equation}
To be rigorous, the relation~\eqref{def:tau}, having no sense when~$\S$ vanishes, should be rewritten as
\begin{equation}
\left\{\begin{aligned}
    & \btau : \S = 2\, Z(\phi, I) p |\S|, \\
    & |\btau| \leq Z(\phi, I) p \quad \text{where $\btau$ is symmetric and trace-less.}
\end{aligned}\right.
\end{equation}
This corresponds to the usual threshold rheology in granular media: the value of~$Z(\phi, I)$ is related to a threshold at which the flow starts to deform. One of the fundamental points is therefore to make~$Z(\phi, I)$ explicit. Currently, only empirical measurements allow us to have access to this threshold and we will see in Section~\ref{sec:examples} several examples of such laws obtained by fitting experimental measurements.

\begin{remark}\label{rem-viscous1}
In~\cite{barker15,barker17,barker23,schaeffer19}, the authors write~$Y(\phi, p, I)$ instead of~$Z(\phi, I)p$. The writing proposed here seems to be relevant since in all models investigated later the stress is proportional to the pressure. Moreover, it allows the conditions that will be proposed later to be written relatively simply.\par
It is important to note that the constitutive law~\eqref{def:tau} excludes the case of purely viscous rheologies, which write $\btau = \eta(\phi,|\S|) \S$.
The addition of such a term to~\eqref{def:tau} does not change the announced results and will be discussed regularly in the following, see Remarks~\ref{rem-viscous2} and~\ref{rem-viscous3}. 
\end{remark}

\subsection{Dilatancy model}

The quantity $\div\,\u$ expresses the evolution of the elementary volumes of fluid under the action of a flow moving at the velocity~$\u$. 
Imposing a zero divergence condition is therefore equivalent to imposing that locally, elementary volumes do not vary. If we want to take into account the expansion/compression phenomena, we must therefore impose an additional law which specifies how the divergence of the velocity field depends on certain quantities which characterise the flow.
One way of expressing the effects of dilation is to impose a relation of the form\footnote{We will see in Remark~\ref{rem:dilatance-geometry} that this choice can be motivated by introducing a dilatancy angle.}
\begin{equation}\label{def:div}
    \div \, \u = 2 |\S| f(\phi, p, I).
\end{equation}
With the constitutive law~\eqref{def:tau}, the dilatancy law~\eqref{def:div} and the pore gas pressure equation~\eqref{pressure-diffusion} added to the conservation of mass and momentum equations (respectively~\eqref{syst1} and~\eqref{syst3}), we have a complete system whose unknowns are the volume fraction~$\phi$, the granular velocity~$\u$ and pressure~$p$, and the pore gas pressure~$p_f$, namely
\begin{align}
& \partial_t \phi + \div ( \phi \, \u ) = 0, \label{syst1b} \\
& \partial_t ((1-\phi) p_f) + \div ( (1-\phi) p_f \u ) + p_{\mathrm{atm}} \div \, \u = p_{\mathrm{atm}} \div ( \kappa(\phi) \nabla p_f ), \label{syst2b} \\
& \phi \rho_s \big( \partial_t \u + \u \boldsymbol{\cdot} \nabla \u \big) = \phi \rho_s \g  - \nabla p + \div \Big( Z(\phi, I) p \frac{\S}{|\S|} \Big)  - \nabla p_f, \label{syst3b} \\
& \div \, \u = 2 |\S| f(\phi, p, I). \label{syst4b}
\end{align}

Note that this system can also be written in non-conservative form. Indeed, introducing the notation $\mathrm d_t = \partial_t + \u\cdot\nabla$ for the convective derivative, equations~\eqref{syst1b}--\eqref{syst4b} rewrite
\begin{align}
& \mathrm d_t \phi + \phi \, \div \, \u = 0, \label{eq:NS1} \\
& (1-\phi) \mathrm d_t p_f + (p_{\mathrm{atm}} + p_f) \div \, \u = p_{\mathrm{atm}} \div ( \kappa(\phi) \nabla p_f ), \\
& \phi \rho_s \mathrm d_t \u = \phi \rho_s \g  - \nabla p + \div \Big( Z(\phi, I) p \frac{\S}{|\S|} \Big)  - \nabla p_f, \label{eq:NS2} \\
& \div \, \u = 2 |\S| f(\phi, p, I). \label{eq:NS3} 
\end{align}

\section{Main properties of the generic model}\label{sec:main-properties}

\subsection{Energy dissipation}\label{ssec:energy}

As announced in the Section~\ref{sec:solid-gaz-model} (see equation~\eqref{energy3}), the quantity $\displaystyle \mathcal E_0(\phi,\u) = \frac{1}{2}\int \phi \rho_s |\u|^2$ satisfies the following equation 
\begin{equation}
    \frac{d}{dt} \big[ \mathcal E_0(\phi,\u) \big] = \int \phi \rho_s \g \cdot \u + A - B + C.
\end{equation}
Moreover, we have seen in the previous section (see equation~\eqref{eq:energieA}), that the term~$A,$ due to the presence of the interstitial gas, satisfies
\begin{equation}\label{eq:energieAbis}
-A = \frac{d}{dt} \big[ \mathcal E_1(\phi,p_f) \big] + \int \kappa(\phi) |\nabla p_f|^2,
\end{equation}
where $\displaystyle \mathcal E_1(\phi,p_f) = \int (1-\phi)H(\rho_f)$, so that the energy equation~\eqref{energy3} rewrites
\begin{equation}
    \frac{d}{dt} \big[ \mathcal E_0(\phi,\u) \big]+\frac{d}{dt} \big[ \mathcal E_1(\phi,p_f) \big] + \int \kappa(\phi) |\nabla p_f|^2 + B - C = \int \phi \rho_s \g \cdot \u.
\end{equation}
Let us now examine the effects of the rheology and the dilatancy law, specified by~\eqref{def:tau} and~\eqref{def:div} respectively, on the energy of the system~\eqref{syst1b}-\eqref{syst4b}.  Replacing the deviatoric stress tensor $\btau$  and the divergence of the velocity field by the relations (21) and (22) in the terms $B$ and $C$ appearing in the energy equation~\eqref{energy3}, we obtain
\begin{equation}
    B-C = \int \btau : \nabla \u - p\, \div \, \u = 2\int (Z-f)p|\S|.
\end{equation}
Reporting this relation in~\eqref{energy3}, we deduce that the total energy defined by $\mathcal E = \mathcal E_0(\phi,\u) + \mathcal E_1(\phi,p_f)$ satisfies the following equation
\begin{equation}
    \frac{d\mathcal E}{dt} + \mathcal D = \int \phi \rho_s \g \cdot \u,
\end{equation}
where
\begin{equation}\label{dissipation-general-exp}
\mathcal D = \int \kappa(\phi) |\nabla p_f|^2 + 2\int (Z(\phi, I)-f(\phi, p, I)) p |\S|
\end{equation}
is a dissipation rate. Indeed, if $Z(\phi, I)-f(\phi, p, I)\ge 0$, the total energy of the fluidised granular model decreases over time in the absence of any external force, such as gravitational force.
The positivity of the dissipation rate, which is expressed by the condition
\begin{equation}\label{dissipation-general}
\text{Dissipation condition:} \quad Z \geq f,
\end{equation}
is one of the main properties of the fluidised granular models proposed in the following, and can now be stated.

\begin{theorem}\label{th1}
Under the dissipation condition~\eqref{dissipation-general} and without external force, the model~\eqref{syst1b}--\eqref{syst4b} has a decreasing energy.
\end{theorem}

\begin{remark}\label{rem-viscous2}
The addition of a viscous term, as suggested in Remark~\ref{rem-viscous1}, does not affect this result. On the contrary, adding a viscous contribution $\eta(\phi,|\S|)\S$ to the stress induces the term
\begin{equation}
\mathcal D_{\mathrm{dis}} = 2\int \eta(\phi, |\S|)|\S|^2 >0.
\end{equation}
in the left-hand side of the energy equation. In this case, the dissipation rate is enhanced.
\end{remark}

\subsection{Consistency of the dilatation law at equilibrium}\label{ssec:equilibrium}

According to the arguments of Pailha and Pouliquen in~\cite[section 3.3]{pailha09} and those of Roux and Radjaï in~\cite{roux98}, 
the rate of volume change, which is proportional to the function~$f$, is related to a deviation from an equilibrium steady-state of the granular medium, which is isochoric, \textit{i.e.} with constant volume.
It has been experimentally observed that at this equilibrium, the volume fraction~$\phi$ depends linearly on the inertial number~$I$ and that an equilibrium relation of the following form
\begin{equation}\label{phi(I)}
    \phi_{\mathrm{eq}}(I) = \phi_{\mathrm{max}} - \Delta\phi \, I
\end{equation}
can be found. The parameters $\phi_{\mathrm{max}}$ and~$\Delta\phi$ are obtained by fitting~\eqref{phi(I)} with experimental measurements (e.g. $\phi_{\mathrm{max}}=0.6$ and $\Delta\phi=0.2$ in~\cite{forterre08}).
Note that DEM (Discrete Element Methods) simulations confirm this result (see~\cite{schaeffer19} for example).
The relation~\eqref{phi(I)} can be used to determine the volume fraction at equilibrium by setting $\phi=\phi_{\mathrm{eq}}(I)$. From this empirical law and the definition of the inertial number~\eqref{def:I}, a state equation for the granular flow can easily be derived (see for instance~\cite{heyman_2017}), which expresses the pressure as
\begin{equation}
    p=\rho_s\left(\frac{d\Delta\phi |\S|}{\phi_{\mathrm{max}}-\phi}\right)^2.
\end{equation}
It is therefore attractive to supplement the conservation equations of mass and momentum with the $\mu(I)$--rheology and the above state law, so that a closed compressible model is obtained. However, such system  based on the $\mu(I),\, \phi(I)$--rheology is always ill-posed in the two-dimensional case as proved by Heyman \textit{et al} in~\cite{heyman_2017} (see also~\cite{schaeffer19}).
Barker~\textit{et al} in~\cite{barker17} followed by Schaeffer~\textit{et al} in~\cite{schaeffer19} developed another approach: a dilatancy function $f$ specifies the rate of volume change through the equation~\eqref{syst4b}. 

Although the empirical law~\eqref{phi(I)} is not prescribed in the fluidised granular model~\eqref{syst1b}--\eqref{syst4b}, nor in the compressible models of Barker~\textit{et al}~\cite{barker17} and Schaeffer~\textit{et al}~\cite{schaeffer19}, it must be used to impose conditions on the function $f$ so that the dilatancy law~\eqref{syst4b} is consistent with the physics of dense granular medium. This is what we aim to do in what follows by introducing equilibrium conditions on $f$.

\begin{remark}
There exists other formulations approaching this relationship between~$I$ and $\phi_{\mathrm{eq}}$. For instance, in~\cite[p.929]{schaeffer19}, the authors suggest
$$\phi_{\mathrm{eq}}^{\mathrm{sch}}(I) = \phi_{\mathrm{max}} - \frac{\Delta\phi}{1+ 1/I},$$
arguing that this law would prevent $\phi$ from becoming negative for large values of~$I$.
In practice, we will show in Subsection~\ref{ssec:phi-bounds} that, independently of this law, the volume fraction solution of the mass conservation equation always remains positive, which makes the more complex law proposed in~\cite{schaeffer19} useless.

Nevertheless, as this law of equilibrium is empirical, experiments make it possible to refine the relationship and improve its accuracy in some cases. Hence, in~\cite[p.3]{robinson23}, the authors propose
\begin{equation}
    \phi_{\mathrm{eq}}^{\mathrm{rob}}(I) = \phi_{\mathrm{max}} - A\, I^a,
\end{equation}
with $A = 0.1305$ and $a = 0.8156$, whereas in~\cite[p.10]{breard22}, the following law
\begin{equation}
    \phi_{\mathrm{eq}}^{\mathrm{bre}}(I) = \frac{\phi_{\mathrm{max}}}{1+I}.
\end{equation}
is suggested. However, in the present article, we will always use the relation~\eqref{phi(I)} but any other reasonable choice could be considered.
\end{remark}

In~\cite{roux98}, Roux and Radjaï proposed a dilatancy model inspired by critical state mechanics that relates the rate of volume change to the deviation from equilibrium, \textit{i.e.}
\begin{equation}\label{roux_radjai}
\div \, \u=2a|\S|(\phi-\phi_{\mathrm{eq}}(I))
\end{equation}
where $a$ is a constant. This relation expresses how two layers of beads confined at a given pressure should expand when subjected to a constant shear rate. Written in the form~\eqref{roux_radjai}, it is clear, that when a deformation occurs (\textit{i.e} $|\S|>0$), if the volume fraction exceeds the equilibrium fraction, then $\div \, \u$ is positive, so the material dilates. On the other hand, if $\phi<\phi_{\mathrm{eq}}(I)$, we have $\div \, \u<0$ which induces a contraction of the granular medium. The granular flow is isochoric ($\div \, \u=0$) when there is no deformation (\textit{i.e} $|\S|=0$).
This behavior is essential and must be reproduced by the generic dilatancy law~\eqref{syst4b}.
Note that the Roux and Radjaï's model~\eqref{roux_radjai} fits in the general framework proposed by Barker \textit{et al}~\cite{barker17} and used here, with the function~$f$ defined as: $f(\phi,p,|\S|)=a(\phi-\phi_{\mathrm{eq}}(I))$.

The empirical law~\eqref{phi(I)} can be equivalently written
\begin{equation}\label{I(phi)}
    I_{\mathrm{eq}}(\phi) =  \frac{\phi_{\mathrm{max}}-\phi}{\Delta\phi},
\end{equation}
which defines the inertial number at the equilibrium by $I=I_{\mathrm{eq}}(\phi).$  
As we will see in Section~\ref{sec:examples}, it is preferable, when we place ourselves in the framework of the linear stability theory developed by Barker et al in~\cite{barker17}, to work with this alternative form and to define the equilibrium conditions in terms of deviation of the inertial number $I$ from the equilibrium state characterized by $I_{\mathrm{eq}}(\phi)$.
We now enonce those conditions which ensure that the dilatancy function $f$ is consistent with the Roux and Radjaï's model~\cite{roux98}. 
\begin{theorem}\label{th2}
Under the following equilibrium conditions
\begin{equation}\label{equilibrium-condition}
\text{Equilibrium conditions:}\quad  
\left\{ \begin{aligned}
    & f(\phi,p,I_{\mathrm{eq}}(\phi))=0,\\
    & f(\phi,p,I)>0 \quad \text{for $I>I_{\mathrm{eq}}(\phi)$}, \\
    & f(\phi,p,I)<0 \quad \text{for $I<I_{\mathrm{eq}}(\phi)$},
\end{aligned} \right.
\end{equation}
the model~\eqref{syst1b}--\eqref{syst4b} is consistent with the physics described by Roux and Radjaï in~\cite{roux98}.
\end{theorem}

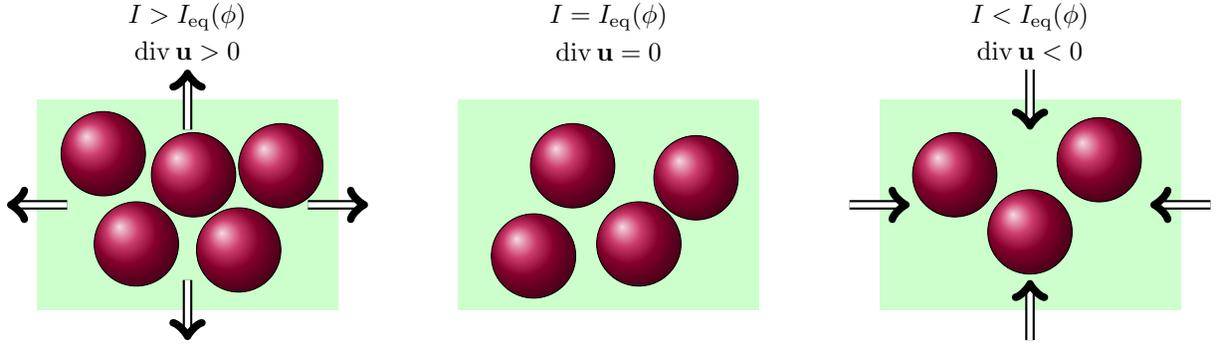
\begin{figure}
\begin{center}
\begin{tikzpicture}[scale=0.4]
\fill  [color=green!20](0,1)--(10,1)--(10,8)--(0,8)--(0,1)--cycle;
\draw  [black, ball color=purple] (2.2,6.2) circle (1.4);
\draw  [black, ball color=purple] (3.3,3.2) circle (1.4);
\draw  [black, ball color=purple] (5.2,5.5) circle (1.4);
\draw  [black, ball color=purple] (6.7,3.0) circle (1.4);
\draw  [black, ball color=purple] (8.1,5.8) circle (1.4);
\draw  [->, double distance = 2pt, thick] (5,7) -- (5,9);
\draw  [->, double distance = 2pt, thick] (5,2) -- (5,0);
\draw  [->, double distance = 2pt, thick] (1,4.5) -- (-1,4.5);
\draw  [->, double distance = 2pt, thick] (9,4.5) -- (11,4.5);
\draw  (5,9) node[above] {$\div \, \u > 0$};
\draw  (5,10) node[above] {$I>I_{\mathrm{eq}}(\phi)$};
\fill  [color=green!20](14,1)--(24,1)--(24,8)--(14,8)--(14,1)--cycle;
\draw  [black, ball color=purple] (16.5,2.8) circle (1.4);
\draw  [black, ball color=purple] (17.8,5.8) circle (1.4);
\draw  [black, ball color=purple] (20,3.2) circle (1.4);
\draw  [black, ball color=purple] (21.9,5.4) circle (1.4);
\draw  (19,9) node[above] {$\div \, \u = 0$};
\draw  (19,10) node[above] {$I=I_{\mathrm{eq}}(\phi)$};
\fill  [color=green!20](28,1)--(38,1)--(38,8)--(28,8)--(28,1)--cycle;
\draw  [black, ball color=purple] (30.5,5.5) circle (1.4);
\draw  [black, ball color=purple] (33.0,3.6) circle (1.4);
\draw  [black, ball color=purple] (35.3,6.0) circle (1.4);
\draw  [<-, double distance = 2pt, thick] (33,7) -- (33,9);
\draw  [<-, double distance = 2pt, thick] (33,2) -- (33,0);
\draw  [<-, double distance = 2pt, thick] (29,4.5) -- (27,4.5);
\draw  [<-, double distance = 2pt, thick] (37,4.5) -- (39,4.5);
\draw  (33,9) node[above] {$\div \, \u < 0$};
\draw  (33,10) node[above] {$I<I_{\mathrm{eq}}(\phi)$};
\end{tikzpicture}
\end{center}
\caption{Change in volume of a granular medium near an isochoric equilibrium state characterized by~$I=I_{\mathrm{eq}}(\phi),$ or equivalently $\phi=\phi_{\mathrm{eq}}(I)$. }
\label{fig:equilibrium}
\end{figure}

The conditions~\eqref{equilibrium-condition} describe the same behaviour of the granular material, when it deviates from the isochoric equilibrium position, as the Roux and Radjaï dilatancy model. Indeed, at equilibrium we have $I=I_{\mathrm{eq}}(\phi)$ so that to the divergence of the velocity field equals zero. 
If $I>I_{\mathrm{eq}}(\phi)$ then the rate of volume change is positive, resulting in an expansion of the granular flow. On the other hand, when the inertial number is less than the equilibrium value, the divergence of the velocity field is negative, resulting in compression of the material. Figure~\ref{fig:equilibrium} illustrates this behaviour.
Note that the non-isochoric and fluidized granular models derived in Section~\ref{sec:examples} will all satisfy the equilibrium conditions~\eqref{equilibrium-condition}.

\subsection{Linear stability of the model}\label{ssec:stability}

According to the work of Barker \textit{et al}~\cite{barker17}, we know that the granular model~\eqref{syst1b}, \eqref{syst3b} with $p_f=0$ and~\eqref{syst4b} is linearly stable as soon as the functions~$Z$ and~$f$ satisfy the three following properties
\begin{align}
    & \text{First stability condition:}  \hspace{-3.3cm}
    && Z - \frac{I}{2}\partial_I Z = f+ I\partial_I f. \label{C1}\\
    & \text{Second stability condition:}  \hspace{-3.3cm}
    && Z + I\partial_I Z \geq 0. \label{C2}\\
    & \text{Third stability condition:}  \hspace{-3.3cm}
    && \partial_p f - \frac{I}{2p}\partial_I f < 0. \label{C3}
\end{align}

The purpose of this section is to show that accounting for the presence of the pore gas in the granular medium as achieved in Section~\ref{sec:fluidisation} by supplementing the mass and momentum conservation equations of the solid particles with the equation~\eqref{syst2b} and adding the term $-\nabla p_f$ in the right-hand side of~\eqref{syst3b}, does not change the stability result stated in~\cite{barker17}, as shown below.

\begin{theorem}\label{th3}
Under the conditions~\eqref{C1}, \eqref{C2} and~\eqref{C3},  the model~\eqref{syst1b}--\eqref{syst4b} is linearly stable.
\end{theorem}

\begin{proof}
To prove this result, we adopt the ideas of the proof made in~\cite[Section 3]{barker17} or in~\cite[section 3]{barker23}. We must estimate the effects of the additional equation describing the evolution of the pore gas pressure, namely~\eqref{syst2b}, and its coupling with the momentum conservation equation through the gradient term~$-\nabla p_f$ in the right-hand side of~\eqref{syst3b}.

Let us consider a solution~$\V^0 = (\phi^0,\u^0,p^0,p_f^0)$ of the system of equations~\eqref{syst1b}--\eqref{syst4b}.
The first step is to linearize this system around $\V^0$ by looking for a perturbed solution in the form $\V = \V^0 + \tilde{\V}$.
As in~\cite{barker17}, we retain only the terms that are linear in the perturbation~$\widetilde{\V}$ and neglect most of the terms that are not of maximal order.
As an example, the linearized version of equation~\eqref{syst2b} describing the evolution of~$p_f$ writes
\begin{equation}\label{pressure-linear}
    \partial_t \widetilde{p_f} = - c \Delta \widetilde{p_f} \qquad \text{where} \;\; c=p_{\mathrm{atm}}\frac{\kappa(\phi^0)}{1-\phi^0}.
\end{equation}
In the next step, the coefficients in the resulting linear system are frozen and we look for exponential solutions $\widetilde{\V}(t,x) = \mathrm{e}^{\mathrm{i}\xi \cdot x + \lambda t}\widehat{\V}$, in order to obtain an eigenvalue problem that can be written as $\lambda \widehat{\V} = \mathcal M\widehat{\V}$.

By specifying the unknown $p_f$, \textit{i.e.}   by decomposing $\widehat{\V} = \begin{pmatrix} \widehat{\U} , \widehat{p}_f \end{pmatrix}$, the matrix~$\mathcal M$ takes the following form
\begin{equation}
    \mathcal M = \begin{pmatrix}
  \phantom{\Huge{|}} \mathcal N
  & \hspace*{-\arraycolsep}\vline\hspace*{-\arraycolsep} & \begin{aligned} 0 \\ \mathrm i \xi \\ 0 \end{aligned} \\
\hline
  \phantom{\Big(} 0 \quad 0 \quad 0 & \hspace*{-\arraycolsep}\vline\hspace*{-\arraycolsep} &
  c|\xi|^2
\end{pmatrix},
\end{equation}
where $\mathcal{N}$ exactly corresponds to the matrix obtained when the fluidisation by the pore gas pressure is not modeled. The term~$\mathrm i \xi$ comes from the pressure gradient~$\nabla p_f$ involved in the momentum equation~\eqref{syst3b} whereas the last line of the matrix~$\mathrm{M}$ corresponds to the equation~\eqref{pressure-linear}.

In~\cite{barker17}, it is proved that the conditions~\eqref{C1}, \eqref{C2} and~\eqref{C3} imply that all eigenvalues of the matrix~$\mathcal{N}$ are positive. Since $c|\xi|^2\geq 0$, these conditions are also sufficient in the present case, \textit{i.e.} with the additional pressure~$p_f$.
\end{proof}

\begin{remark}
The second condition~\eqref{C2} is not exactly the same as the one given in~\cite{barker17}, which is stronger since it requires a strict inequality $\partial_I Z > 0$.
Looking in details at the proof in~\cite{barker17}, it is clear that the result remains valid if $I\partial_I Z + Z \geq 0$.
Indeed, this condition is used to show that the trace of a matrix, namely 
$$\mathrm{Tr}(\mathrm{M}+\mathrm{N}) = \Big( \frac{I\partial_I Y}{2\|D^\star\|} +\frac{Y}{2\|D^\star\|} \Big) |\xi|^2 + \frac{1}{\Gamma} \Big( (1+B)^2 \xi_1^2 + (1-B)^2 \xi_2^2 \Big),$$
is strictly positive. Note the above equation is written with the notations used in~\cite[Proof of lemma 4.1]{barker17}.
Since $\Gamma>0$, according to~\eqref{C3}, the second term in the right-hand side of the above equality is always strictly positive. Therefore,
the condition $I\partial_IY+Y\geq 0$ is sufficient to ensure the positiveness of the trace.
Note that the conditions for the linear stability of the model established in~\cite{barker17} are only sufficient conditions. There is no evidence that they are optimal.
\end{remark}

\subsection{Volume fraction bounds}\label{ssec:phi-bounds}

In order for the model to make sense, it is fundamental to ensure that the volume fraction~$\phi$ always remains positive, and does not exceed the maximum value~$\phi_{\mathrm{max}}$ previously introduced. We are now stating one of the main results of this paper.
\begin{theorem}\label{th4}
Suppose that the initial condition satisfies $0 \leq \phi|_{t=0} \leq \phi_{\mathrm{max}}$ and consider any smooth solution to~\eqref{syst1b}--\eqref{syst4b}.
We have
$$\phi\geq 0.$$
Moreover if the consistency conditions~\eqref{equilibrium-condition} are satisfied then we have
$$\phi \leq \phi_{\mathrm{max}}.$$
\end{theorem}

\begin{proof} The proof is based on the following observation:\par
If $(\phi,\u)$ are regular and satisfy the equation $\partial_t \phi + \div ( \phi \u ) = 0$ then for any function $\beta\in \mathcal C^1(\mathbb R,\mathbb R)$, we have
\begin{equation}\label{dt(beta)}
    \partial_t \beta(\phi) + \div ( \beta(\phi) \u ) + (\phi \beta'(\phi) - \beta(\phi)) \, \div \, \u = 0.
\end{equation}
\textit{Step 1: Positiveness of $\phi$.} In order to show that~$\phi$ remains positive as soon as $\phi|_{t=0}\geq 0$, we will use the above equation with the functions~$\beta_\varepsilon$ defined, for all $\varepsilon>0$, by
\begin{equation}
    \beta_\varepsilon(\phi) = \left\{\begin{aligned} 
    & 0 && \text{if $\phi \geq \varepsilon$},\\
    & -\phi && \text{if $\phi \leq -\varepsilon$},\\
    & \frac{1}{4\varepsilon}(\phi-\varepsilon)^2 && \text{if $|\phi| \leq \varepsilon$}.
    \end{aligned}\right.
\end{equation}
It is easy to show that $\beta_\varepsilon$ is of class~$\mathcal C^1$ and that
\begin{equation}
    \phi \beta_\varepsilon'(\phi) - \beta_\varepsilon(\phi) = \left\{\begin{aligned} 
    & 0 && \text{if $|\phi| \geq \varepsilon$},\\
    & \frac{1}{4\varepsilon}(\phi^2-\varepsilon^2) && \text{if $|\phi| \leq \varepsilon$}.
    \end{aligned}\right.
\end{equation}
In particular, we have $\displaystyle |\phi \beta_\varepsilon'(\phi) - \beta_\varepsilon(\phi)| \leq \frac{\varepsilon}{4}$ so that after integrating the equation~\eqref{dt(beta)}, we obtain
\begin{equation}
    \frac{d}{dt} \int \beta_\varepsilon(\phi) \leq \frac{\varepsilon}{4}\int |\div \, \u|.
\end{equation}
By passing to the limit when $\varepsilon$ tends to $0$, we deduce that
\begin{equation}
    \frac{d}{dt} \int \beta(\phi) \leq 0
    \quad \text{where} \quad 
    \beta(\phi) = \left\{\begin{aligned} 
    & 0 && \text{si $\phi \geq 0$},\\
    & -\phi && \text{si $\phi \leq 0$}.
    \end{aligned}\right.
\end{equation}
The condition $\phi|_{t=0}\geq 0$ means that $\beta(\phi|_{t=0})=0$. Since $\int \beta(\phi)$ decreases with time, and is by definition always positive, we deduce that, for any time, $\beta(\phi)=0$ which means that $\phi \geq 0$.\\

\textit{Step 2: Upper bound on $\phi$.} In general, there is no reason why a solution~$\phi$ of the equation $\partial_t \phi + \div ( \phi \u ) = 0$ should remain bounded if the divergence of the field~$\u$ is not zero. Here, we will use the following additional information which comes from the equilibrium conditions~\eqref{equilibrium-condition}. The equation~\eqref{syst4b} specifying the divergence of the velocity field implies:
\begin{equation}\label{0933}
    \phi \geq \phi_{\mathrm{max}} \quad \Longrightarrow \quad \div \, \u \geq 0.
\end{equation}
Indeed, if $\phi \geq \phi_{\mathrm{max}}$ then $\phi \geq \phi_{\mathrm{max}} - \Delta\phi \, I$ which is equivalent to $I \geq I_{\mathrm{eq}}(\phi)$. According to the condition~\eqref{equilibrium-condition} we have, in this case, $f \geq 0$ which implies $\div \, \u \geq 0$.

Next, we use the same argument as for showing the positiveness of~$\phi$ but with the function~$\beta$ defined by
\begin{equation}
    \beta(\phi) = \left\{\begin{aligned} 
    & 0 && \text{if $\phi \leq \phi_{\mathrm{max}}$},\\
    & (\phi-\phi_{\mathrm{max}})^2 && \text{if $\phi \geq \phi_{\mathrm{max}}$}.
    \end{aligned}\right.
\end{equation}
This function is of class~$\mathcal C^1$ and we have
\begin{equation}
    \phi\beta'(\phi)-\beta(\phi) = \left\{\begin{aligned}
    & \qquad 0 && \text{if $\phi \leq \phi_{\mathrm{max}}$},\\
    & \phi^2-\phi_{\mathrm{max}}^2 && \text{if $\phi \geq \phi_{\mathrm{max}}$}.
    \end{aligned}\right.
\end{equation}
Using this expression and the implication~\eqref{0933}, we notice that $(\phi \beta'(\phi)-\beta(\phi)) \div \, \u \geq 0$ so that the integration with respect to the space variables of the equation~\eqref{dt(beta)} implies
\begin{equation}
    \frac{d}{dt} \int \beta(\phi) \leq 0.
\end{equation}
Thus, if we have $\phi|_{t=0} \leq \phi_{\mathrm{max}}$ then $\beta(\phi|_{t=0})=0$ and $\beta(\phi)=0$ so that the result follows, namely $\phi \leq \phi_{\mathrm{max}}$.
\end{proof}
\section{Physical examples}\label{sec:examples}

The objective of this section is to complete the generic fluidised~\eqref{syst1b}-\eqref{syst4b} model presented in Section 4, with constitutive laws specified by the yield function $Z$ and the dilatancy function $f$.
The fundamental question is how to determine~$Z$ and~$f$ in such a way that the physics of granular flows is properly described and the mathematical properties introduced in the previous section are verified, precisely the energy of the system must be dissipated, the volume fraction must be positive and upper bounded, and the model must be linearly stable.
These last three properties are verified if the assumptions of Theorems~\ref{th1},~\ref{th3} and~\ref{th4} are fulfilled, which is the case if the dissipation condition~\eqref{dissipation-general}, the stability conditions~\eqref{C1}--\eqref{C3} and the equilibrium consistency condition~\eqref{equilibrium-condition} are satisfied.

Note that due to the Theorem~\ref{th2}, this last condition requires that the sign of the function $f$ corresponds to a deviation of the granular flow from an isochoric equilibrium, \textit{i.e.} of constant volume. The expansion or contraction of the granular material near this equilibrium state is then taken into account in the model. The function~$f$ must therefore fulfil the condition~\eqref{equilibrium-condition} for both physical and mathematical reasons.

The question of the choice of rheology, \textit{i.e.} the choice of the function $Z$, must now be addressed. This choice is guided by the state of knowledge on the physics of granular materials. The model classically used is the Drucker-Prager model, which is a multidimensional form of the Mohr-Coulomb law (see for example~\cite{forterre08}). Later, the $\mu(I)$--rheology emerged thanks to the work of the Groupement de Recherche Milieux Divisés~\cite{gdrmidi} and became established. 
In practice, these constitutive laws defining the deviatoric stress tensor are obtained by fitting experimental measurements and can therefore be used as a starting point for the construction of a model.

Our methodology is as follows: once the yield function~$Z$ is prescribed according to the criteria discussed above, the stability condition~\eqref{C1}, which can be reduced to a differential equation to the function~$f$, is used to determine~$f$. At this stage, $f$ is known up to a constant which is independent on the inertial number but may depend on the volume fraction $\phi$ and the pressure $p$. This constant will be fixed so that at equilibrium, the divergence of the velocity field equals zero in agreement with the critical state theory proposed by Roux and Radjaï in~\cite{roux98}.
The last step will be to make sure that the other conditions~\eqref{dissipation-general}, \eqref{C2}, \eqref{C3} and~\eqref{0933} are fulfilled. Following this approach, we derive hereafter several fluidised and non-isochoric granular models.

Before going into the details of the choice of the constitutive law and the determination of the rate of volume change, let us mention that the equation~\eqref{C1} is a linear in~$Z$ and~$f$. 
Thus, any linear combination of pairs of functions $(Z,f)$ verifying~\eqref{C1} will also be a solution of~\eqref{C1}. 
Moreover, as no derivative with respect to the variable~$\phi$ appears in the calculation of~$f$ (see the condition~\eqref{C1}), 
the coefficients of such linear combinations may depend on~$\phi$. This observation will facilitate the developments presented in what follows.

Finally, it should be mentioned that relatively simple models should be preferred as long as a numerical scheme can be designed and implemented. However, this is beyond the scope of this paper and will be the subject of future work.

\subsection{Stabilisation of the Drucker-Prager model using a divergence condition}\label{ssec-DP1}

Granular media have been studied extensively (see for instance~\cite{forterre08} and the references therein) and several rheologies have been proposed. It is therefore useful to have in mind the usual forms that the constitutive law, defining the function $Z$, can take in the literature. One of the first choices comes from solid mechanics and is known as Drucker-Prager model. Following the book of Andreotti \textit{et al}~\cite[Table 4.1]{andreotti12} the deviatoric stress tensor is described by
\begin{equation}
    \btau = \sin (\delta)\, p \frac{\S}{|\S|},
\end{equation}
where~$\delta$ corresponds to the internal angle of friction. Note that this rheology is most often written using the tangent function instead of the sine function, but with a different angle (see~\cite[p.144]{andreotti12} for the details).

In this case, the function~$Z$ previously introduced reads
\begin{equation}
    Z(\phi,I) = \sin (\delta).
\end{equation}
The stability condition~\eqref{C1} implies that $\partial_I(If(\phi,p,I)) = \sin (\delta)$. Integrating this equation, we deduce that there exists a function~$K$ independent of $I$ such that $If(\phi,p,I) = \sin (\delta) \, I + K(\phi,p)$.
In order to determine the function~$K$, we invoke the consistency condition at equilibrium~\eqref{equilibrium-condition} and we obtain
\begin{equation}\label{f-DP}
    f(\phi,p,I) = \sin (\delta) \Big( 1 - \frac{I_{\mathrm{eq}}(\phi)}{I} \Big).
\end{equation}
Noting that $I_{\mathrm{eq}}(\phi)$ and $I$ are positive, it immediately follows that $Z \geq f$ so that the
dissipation condition~\eqref{dissipation-general} holds.
The equilibrium condition~\eqref{equilibrium-condition} and the first stability condition~\eqref{C1} are also fulfilled by construction of~$f$ from~$Z$. The two other stability conditions~\eqref{C2} and~\eqref{C3}, which in this particular case reads $Z + I\partial_I Z \geq 0$ and $\partial_I f >0$, are obvious.

By using the definition~\eqref{def:I} of the inertial number~$I$ and the expression of~$f$ found above~\eqref{f-DP}, the dilatancy law~\eqref{def:div} can be rewritten 
\begin{equation}\label{def:div-bis}
    \div \, \u = 2 \sin(\delta) |\S| - \frac{2 \sin(\delta)}{d \sqrt{\rho_s}} I_{\mathrm{eq}}(\phi) \sqrt{p}.
\end{equation}
From a physical point of view, this relationship highlights a competition between shear effects which tend to increase the volume (solid particles separate when the material is sheared) and pressure effects which compress the flow. Figure 2 illustrates these effects on the granular medium.
\begin{figure}
\begin{center}
\begin{tikzpicture}[scale=0.4]
\fill  [color=green!20](0,-0.05)--(10,-0.05)--(10,8.05)--(0,8.05)--(0,-0.05)--cycle;
\draw  [thick] (0,-0.05)-- (10,-0.05);
\draw  [thick] (0,8.05)-- (10,8.05);
\draw  [black, ball color=purple] (5,4) circle (1.5);
\draw  [black, ball color=purple] (6.7,6.5) circle (1.5);
\draw  [black, ball color=purple] (3.3,1.5) circle (1.5);
\draw  [<-, double distance = 2pt, thick] (2,8.05) -- (5,8.05);
\draw  [->, double distance = 2pt, thick] (5,-0.05) -- (8,-0.05);
\draw  [very thick,->](15,7.5) arc (35:125:4);
\draw (12,9) node[below]{Shear} ;
\draw (12,8) node[below]{stress} ;
\fill  [color=green!20](14,0.95)--(24,0.95)--(24,7.05)--(14,7.05)--(14,0.95)--cycle;
\draw  [thick] (14,0.95)-- (24,0.95);
\draw  [thick] (14,7.05)-- (24,7.05);
\draw  [black, ball color=purple] (19,4) circle (1.5);
\draw  [black, ball color=purple] (21.6,5.5) circle (1.5);
\draw  [black, ball color=purple] (16.4,2.5) circle (1.5);
\draw  [very thick,<-](29,6.5) arc (35:125:4);
\draw (26,7.7) node[below]{Normal} ;
\draw (26,6.7) node[below]{stress} ;
\fill  [color=green!20](28,1.95)--(38,1.95)--(38,6.05)--(28,6.05)--(28,1.95)--cycle;
\draw  [thick] (28,1.95)-- (38,1.95);
\draw  [thick] (28,6.05)-- (38,6.05);
\draw  [black, ball color=purple] (33,4) circle (1.5);
\draw  [black, ball color=purple] (36.0,4.5) circle (1.5);
\draw  [black, ball color=purple] (30.0,3.5) circle (1.5);
\draw  [<-, double distance = 2pt, thick] (33,6.05) -- (33,8.05);
\draw  [->, double distance = 2pt, thick] (33,-0.05) -- (33,1.95);
\draw  (5,10) node[above] {$\div \, \u > 0$};
\draw  (5,9) node[above] {volume increases};
\draw  (33,10) node[above] {$\div \, \u < 0$};
\draw  (33,9) node[above] {volume decreases};
\end{tikzpicture}
\end{center}
\caption{Volume change under shear (left) and normal (right) stresses.}
\label{fig:stress_effect}
\end{figure}
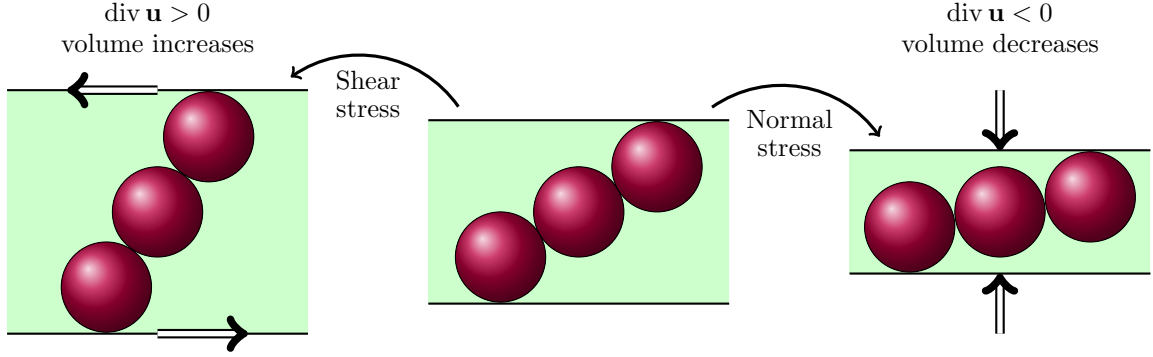

Using the definition of the equilibrium via~$I_{\mathrm{eq}}(\phi)$ or via~$\phi_{\mathrm{eq}}(I)$, see~\eqref{I(phi)} or~\eqref{phi(I)}, it is possible to rewrite the function~$f$ obtained in~\eqref{f-DP} as follows
\begin{equation}
    f(\phi,p,I) = \frac{\sin(\delta)}{\Delta\phi \, I} \big( \phi - \phi_{\mathrm{eq}}(I) \big).
\end{equation}
Written in this form, the dilatancy model obtained is close to the one proposed by Roux and Radjaï~\cite{roux98} for which we have 
\begin{equation}\label{f-RR}
f(\phi,p,I) = a \big( \phi - \phi_{\mathrm{eq}}(I) \big).
\end{equation}

We are now able to write a fluidized non-isochoric granular model based on the Drucker-Prager rheology, which is linearly stable and physically consistent, namely
\begin{align}
& \partial_t \phi + \div ( \phi \, \u ) = 0, \label{syst1DP} \\
& \partial_t ((1-\phi) p_f) + \div ( (1-\phi) p_f \u ) + p_{\mathrm{atm}} \div \, \u = p_{\mathrm{atm}} \div ( \kappa(\phi) \nabla p_f ), \label{syst2DP} \\
& \phi \rho_s \big( \partial_t \u + \u \boldsymbol{\cdot} \nabla \u \big) = \phi \rho_s \g  - \nabla p + \div \Big( \sin (\delta) p \frac{\S}{|\S|} \Big)  - \nabla p_f, \label{syst3DP} \\
& \div \, \u = 2 \sin (\delta) |\S| - 2 \lambda \sin (\delta) (\phi_{\mathrm{max}}-\phi)\sqrt{p}, \label{syst4DP}
\end{align}
where $\lambda = \frac{1}{\Delta\phi d \sqrt{\rho_s}}$.
Moreover, as we have previously seen in the Subsection~\ref{ssec:energy}, the energy for this model is dissipated over time, namely
\begin{equation}
    \frac{d \mathcal E}{dt} + \mathcal D = \int \phi \rho_s \g \cdot \u,
\end{equation}
where
\begin{equation}
\mathcal D = 2 \lambda \sin (\delta) \int (\phi_{\mathrm{max}}-\phi)p\sqrt{p} + \int \kappa(\phi) |\nabla p_f|^2.
\end{equation}
Note that this rheology is known to be linearly unstable when associated with the incompressibility condition~$\div \, \u = 0$, see~\cite[Appendix C]{martin17}. Here, by coupling it with the dilatancy law~\eqref{syst4DP}, we obtain a model satisfying the stability conditions~\eqref{C1}, \eqref{C2} and~\eqref{C3}.

\subsection{Stabilisation of the $\mu(I)$-rheology using a divergence condition}\label{ssec-mu(I)1}

In the class of consitutive laws for granular media, one of them is almost unanimously accepted: this is the $\mu(I)$-rheology. It is basically a kind of generalisation of the Drucker-Prager model in which the effective viscosity depends on the inertial number~$I$. With the notations introduced in this article, this model corresponds to 
\begin{equation}
    Z(\phi,I) = \mu(I),
\end{equation}
where the function~$\mu$ depends on three parameters $I_0>0$, $\mu_2>\mu_1>0$ and is given by
\begin{equation}\label{mu(I)}
    \mu(I) = \mu_1+\frac{\mu_2-\mu_1}{1+I_0/I}.
\end{equation}

Following~\cite[p.10]{forterre08} or~\cite[p.226]{andreotti12}, typical values of the constants obtained for mono-dispersed glass beads are $\mu_1 = \tan (21^o)$, $\mu_2 = \tan (33^o)$ and $I_0 = 0.3$.
Note that when $I$ tends to $0$, the Drucker-Prager model is recovered, which means that $\mu_1=\sin(\delta)$ must be related to the internal friction.
It may also be noted that these values are obtained experimentally, and that some authors specify that~$I_0$ can depend on~$\phi$ (see for instance~\cite[equation (2.12)]{gray14} or~\cite[Appendix A]{jop05}).

Let us now see if it is possible to derive a complete and stable model with this choice of rheology. First of all, the condition~\eqref{C2} is satisfied since $\mu>0$ and $\mu'>0$. In order to ensure that the condition~\eqref{C1} is satisfied, the function~$f$ must be a solution of the differential equation 
\begin{equation}
    \partial_I(If) = \mu(I) - \frac{I}{2}\mu'(I).
\end{equation}
By integrating this ordinary differential equation, with respect to~$I$, and using the condition $f(\phi,p,I_{\mathrm{eq}}(\phi))=0$, we find
\begin{equation}
    f(\phi,p,I) = F(I) - \frac{I_{\mathrm{eq}}(\phi)}{I}F(I_{\mathrm{eq}}(\phi)),
\end{equation}
where $F(I)=\frac{3}{2I}M(I) - \frac{1}{2}\mu(I)$, $M$ being the primitive of~$\mu$ which vanishes at~$I=0$. Using the explicit expression~\eqref{mu(I)}, we derive
\begin{equation}\label{F-mu(I)}
    F(I) = \mu_2 + \frac{\mu_2-\mu_1}{2}H\Big( \frac{I}{I_0} \Big) \quad \text{where} \quad H(x) = \frac{1}{1+x}-\frac{3\ln(1+x)}{x}.
\end{equation}
Finally, imposing the third stability condition~\eqref{C3} amounts to writing
\begin{equation}\label{C3-bis}
    F'(I) + \frac{I_{\mathrm{eq}}(\phi)}{I^2}F(I_{\mathrm{eq}}(\phi)) > 0.
\end{equation}
By observing that $H'(x)>0$ for $x>0$, and that $H(0)=-2$, we deduce that 
\begin{equation}
    F'(I) + \frac{I_{\mathrm{eq}}(\phi)}{I^2}F(I_{\mathrm{eq}}(\phi)) = \frac{\mu_2-\mu_1}{2I_0} H'\Big(\frac{I}{I_0}\Big) + \frac{I_{\mathrm{eq}}(\phi)}{I^2} \Big( \mu_2 + \frac{\mu_2-\mu_1}{2} H\Big(\frac{I}{I_0}\Big) \Big) > 0,
\end{equation}
which means that the condition~\eqref{C3} (equivalent to~\eqref{C3-bis}) is always satisfied.

In the case of the $\mu(I)$-rheology, the energy of the model is also dissipated, as 
\begin{equation}
\mathcal D = 2\int p|\S|(\mu(I)-F(I)) + \frac{2}{d\sqrt{\rho_s}} \int I_{\mathrm{eq}}(\phi)F(I_{\mathrm{eq}}(\phi)) p\sqrt{p}  
\end{equation}
is positive since
\begin{equation}
    \mu(I)-F(I) = \frac{3(\mu_2-\mu_1)}{2} \Big( \frac{\ln(1+I/I_0)}{I/I_0} - \frac{1}{1+I/I_0}\Big) > 0.
\end{equation}

As a consequence, the fluidised and non-isochoric granular model, based on the $\mu(I)$-rheology, satisfies all the required properties and reads
\begin{align}
& \partial_t \phi + \div ( \phi \, \u ) = 0, \label{syst1mu(I)} \\
& \partial_t ((1-\phi) p_f) + \div ( (1-\phi) p_f \u ) + p_{\mathrm{atm}} \div \, \u = p_{\mathrm{atm}} \div ( \kappa(\phi) \nabla p_f ), \label{syst2mu(I)} \\
& \phi \rho_s \big( \partial_t \u + \u \boldsymbol{\cdot} \nabla \u \big) = \phi \rho_s \g  - \nabla p + \div \Big( \mu(I) p \frac{\S}{|\S|} \Big)  - \nabla p_f, \label{syst3mu(I)} \\
& \div \, \u = 2 F(I) |\S|  - 2\widetilde{F}(\phi)\sqrt{p}, \label{syst4mu(I)}
\end{align}
where the functions $F$ and $\widetilde{F}$ are defined by
\begin{align}\label{F-mu(I)}
    & F(I) = \mu_2 + \frac{\mu_2-\mu_1}{2}\Big( \frac{1}{1+I/I_0}-\frac{3\ln(1+I/I_0)}{I/I_0} \Big), \\
    & \widetilde{F}(\phi) = \frac{1}{d\sqrt{\rho_s}}I_{\mathrm{eq}}(\phi)F(I_{\mathrm{eq}}(\phi))\quad \text{with} \quad I_{\mathrm{eq}}(\phi) =  \frac{\phi_{\mathrm{max}}-\phi}{\Delta\phi}.
\end{align}

Of course, when~$I$ goes to~$0$, we recover the Drucker-Prager model, presented in the previous subsection, with $\mu_1=\sin(\delta)$ since
\begin{equation}
    \lim_{I\to 0} \mu(I) = \mu_1 \quad \text{and} \quad 
    \lim_{I\to 0} F(I) = \mu_1.
\end{equation}

\subsection{Stabilisation for a generic rheology using a divergence condition}\label{ssec-mu(I)2}

Since the stability conditions~\eqref{C1}, \eqref{C2} and~\eqref{C3} do not involve a derivative with respect to the volume fraction~$\phi$, it is also possible to propose a law of the~$\mu(I)$--type taking into account the dependence with respect to~$\phi$.
This means that, prescribing $Z(\phi,I)$, integrating the stability condition~\eqref{C1} and using the equilibrium condition~\eqref{equilibrium-condition}, we find
\begin{equation}\label{f-general}
    f(\phi,p,I) = W(\phi,I) - I_{\mathrm{eq}}(\phi) W(\phi,I_{\mathrm{eq}}(\phi)) \frac{1}{I},
\end{equation}
where $W$ is defined from $Z$ by
\begin{equation}
    W(\phi,I) = \frac{3}{2I}\int^I_{I_1} Z(\phi,J)\, \mathrm dJ - \frac{Z(\phi,I)}{2}.
\end{equation}
Note that the choice of the constant~$I_1$ has no influence on the value of the function $f$ since replacing~$W$ by~$W+\frac{3I_1}{2I}$ does not change~$f$.
The choice of~$I_1$ can be used in such a way that the integral makes sense, for example if~$Z$ admits a singularity at a point which one wishes to ignore.

By performing an expansion in the~$\phi$ variable in the vicinity of the equilibrium $\phi=\phi_{\mathrm{eq}}(I)$, it is possible to rewrite~$f$ as
\begin{equation}
    f(\phi,p,I) = \frac{W(\phi_{\mathrm{eq}}(I),I)}{\Delta\phi \, I} \big( \phi - \phi_{\mathrm{eq}}(I) \big) + \mathcal{O}\big(\big( \phi - \phi_{\mathrm{eq}}(I) \big)^2 \big),
\end{equation}
which again is similar to the model~\eqref{f-RR} and therefore takes the same form as the one proposed in~\cite{roux98} with $a=\frac{W(\phi_{\mathrm{eq}}(I),I)}{\Delta\phi \, I}$.

In the Table~\ref{tab:1}, we show the values of~$f$ obtained by the general formula~\eqref{f-general} when polynomial functions in the variable $I$ are chosen for the function~$Z$. Because of the linearity of the $Z\mapsto f$ correspondence, as indicated in the introduction to the Section~\ref{sec:examples}, it is possible to write $Z$ as linear combinations of each row of the table with coefficients depending on~$\phi$.
\begin{table}
\begin{center}
\begin{tabular}{|c|c|c|}
\hline
$Z(\phi,I)$ & $f(\phi,p,I)$ & $\text{Contribution to dissipation }\mathcal D$ \\
\hline\hline
$1$ & $\displaystyle 1 - \frac{I_{\mathrm{eq}}(\phi)}{I}$ & $\displaystyle \frac{2I_{\mathrm{eq}}(\phi)}{d\sqrt{\rho_s}}p\sqrt{p}$ \rule[-15pt]{0pt}{35pt} \\
\hline
$I$ & $\displaystyle \frac{1}{4} \Big( I - \frac{I_{\mathrm{eq}}(\phi)^{2}}{I} \Big)$ & $\displaystyle \frac{3}{4} I + \frac{1}{4}\frac{I_{\mathrm{eq}}(\phi)^{2}}{d\sqrt{\rho_s}}p\sqrt{p}$ \rule[-15pt]{0pt}{35pt} \\
\hline
$I^n,\;(n\neq -1)$ & $\displaystyle \frac{2-n}{2(n+1)} \Big( I^n - \frac{I_{\mathrm{eq}}(\phi)^{n+1}}{I} \Big)$ & $\displaystyle \frac{3n}{2(n+1)} I^n + \frac{2-n}{2(n+1)}\frac{I_{\mathrm{eq}}(\phi)^{n+1}}{d\sqrt{\rho_s}}p\sqrt{p}$ \rule[-15pt]{0pt}{35pt} \\
\hline
\end{tabular}
\end{center}
\caption{Dilatancy functions $f(\phi,p,I)$ for yield functions $Z(\phi,p,I)$ given as powers of the inertial number. The last column shows the corresponding contribution to the dissipation rate $\mathcal D$, in the energy equation.}
\label{tab:1}
\end{table}

\begin{remark}\label{rem-viscous3}
The case $Z(\phi,I) = I^2$, corresponding to the case $n=2$ in the Table~\ref{tab:1}, is particular. Indeed, the corresponding dilatancy function is $f=0$. Using the definition of the inertial number and  the deviatoric stress tensor, such a choice of function~$Z$ corresponds to the following rheology 
\begin{equation}
    \btau = Z(\phi,I) p \frac{\S}{|\S|} = d^2 \rho_s |\S| \S,
\end{equation}
which is purely viscous.
More generally, viscous rheology may be written as
\begin{equation}\label{eq:viscous-contribution}
    \btau = \eta(\phi,|\S|) \S.
\end{equation}
Such rheologies are not treated in the present study as noted in the Remarks~\ref{rem-viscous1} and~\ref{rem-viscous2}.
However, the stability condition equivalent to~\eqref{C1} in~\cite[p.5]{barker17}, shows that $f=0$ is compatible with~\eqref{eq:viscous-contribution}. It will therefore be possible to add any viscous contribution to the stress without changing the stability conditions.

Note also that the case where $Z=I^n$ with $n>2$ produces a negative dissipation term. It is therefore likely that the choice of rheology containing such terms is not dissipative, and should therefore be excluded.
\end{remark}

\subsection{Rheologies with dilatancy effects}

\subsubsection{The Ducker-Prager rheology with dilatancy effect}
In order to take into account the variations of the volume fraction of the granular medium in the rheology, Wood~\cite{wood90} proposed to modify the Drucker-Prager model by introducing what is called a dilatation angle~$\psi$.
This angle~$\psi$ is a measure of the ratio between the relative vertical and horizontal displacements between two layers of grains when they are sheared. It can be positive (expansion) or negative (contraction).
Intuitively, we can guess that a close packing that has to dilate in order to deform ($\psi > 0$) has a friction coefficient larger than that of a loose packing that will undergo compaction ($\psi < 0$). The granular viscosity is then an increasing function of~$\psi$, and must be equal to $\sin(\delta)$ when the dilatancy angle is zero.
More precisely, the resulting laws, proposed in~\cite[p.150-151]{andreotti12}, are written as 
\begin{align}
    & \btau = \sin (\delta+\psi)\, p \frac{\S}{|\S|}\label{eq:84}\\
    & \div \, \u = 2|\S| \sin (\psi).\label{eq:85}
\end{align}

\begin{remark}\label{rem:dilatance-geometry}
As for the friction angle, some authors use the tangent function instead of the sine function. Specifically, the dilatancy angle is the angle of motion relative to the horizontal arising from displacement, with $dY = \tan(\psi)dX$ where~$dY$ and~$dX$ are the vertical and horizontal displacements. This definition of~$\psi$ is specific to planar shear but can be generalised by the equation~\eqref{eq:85}.
The diagram below represents the dilatancy angle for a two-dimensional flow whose velocity field depends only on the vertical variable~$y$.\\[0.3cm]
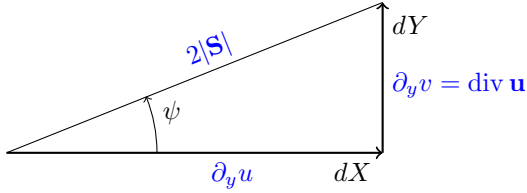
\begin{minipage}{7.5cm}
\begin{center}
\begin{tikzpicture}[scale=1]
\draw  [->, thick] (0,0)-- (5,0);
\draw (4.6,0) node[below]{$dX$} ;
\draw (3,0) node[below, color=blue]{$\partial_y u$} ;
\draw  [->, thick] (5,0)-- (5,2);
\draw (5,1.7) node[right]{$dY$} ;
\draw (5,0.9) node[right, color=blue]{$\partial_y v = \div \, \u$} ;
\draw (0,0)-- (5,2);
\draw (2.3,1.2) node[right,rotate=22, color=blue]{$2|\S|$} ;
\draw  [->](2,0) arc (0:22:2);
\draw (2.2,0.25) node[above]{$\psi$} ;
\end{tikzpicture}
\end{center}
\end{minipage}
\begin{minipage}{8cm}
In a two-dimensional case and assuming that the velocity field is written $\u = (u(y),v(y))$, we have $\div \, \u = \partial_y v$ and
$$\S = \frac{1}{2} \begin{pmatrix} -\partial_y v & \partial_y u \\ \partial_y u & \partial_y v\end{pmatrix} \; \text{ so that } \; 2|\S| = \sqrt{(\partial_y u)^2 + (\partial_y v)^2}.$$
In this two-dimensional case, the definition of~$\psi$ via $\tan(\psi)=\frac{dY}{dX}$ coincides with $\sin(\psi)=\frac{\div \, \u}{2|\S|}$.
\end{minipage}
\par\vspace{0.3cm}
The definition of the dilatancy angle~$\psi$ in the case of a $3$-dimensional flow is different. In~\cite[p.150-151]{andreotti12}, the authors define~$\psi$ by the formula~\eqref{eq:85} by replacing the constant~$2$ by~$3$, while explaining that this definition no longer coincides with that given in the case of simple shear.
\par
Nevertheless, it is possible to do the same reasoning as in dimension~$2$ as shown below.
\par
\vspace{0.3cm}
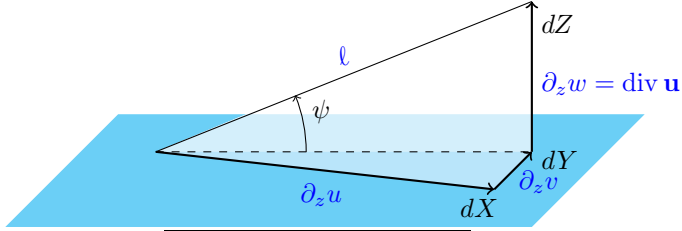
\begin{minipage}{10cm}
\begin{center}
\begin{tikzpicture}[scale=1]
\fill [color=cyan!50] (0,0)--(4.5,-0.5)--(5,0)--(5,0.5)--(6.5,0.5)--(5,-1)--(-2,-1)--(-0.5,0.5)--(1.2,0.5)--cycle ;
\fill [color=cyan!15] (0,0)--(4.5,-0.5)--(5,0)--(5,0.5)--(1.2,0.5)--cycle ;
\fill [color=cyan!25] (0,0)--(4.5,-0.5)--(5,0)--cycle ;
\draw  [dashed] (0,0)--(5,0);
\draw  [->, thick] (0,0)--(4.5,-0.5);
\draw (4.3,-0.5) node[below]{$dX$} ;
\draw (2.2,-0.3) node[below, color=blue]{$\partial_z u$} ;
\draw  [->, thick] (4.5,-0.5)-- (5,0);
\draw (5,-0.1) node[right]{$dY$} ;
\draw (4.7,-0.4) node[right, color=blue]{$\partial_z v$} ;
\draw  [->, thick] (5,0)-- (5,2);
\draw (5,1.7) node[right]{$dZ$} ;
\draw (5,0.9) node[right, color=blue]{$\partial_z w = \div \, \u$} ;
\draw (0,0)-- (5,2);
\draw (2.3,1.2) node[right,rotate=22, color=blue]{$\ell$} ;
\draw  [->](2,0) arc (0:22:2);
\draw (2.2,0.25) node[above]{$\psi$} ;
\end{tikzpicture}
\end{center}
\end{minipage}
\begin{minipage}{7cm}
Assuming that $\u = (u(z),v(z),w(z))$.
\par
We have $\div \, \u = \partial_z w$ and
$$\S = \begin{pmatrix}
\phantom{\Big|\!\!} -\frac{1}{3}\partial_z w & 0 & \frac{1}{2}\partial_z u \\
\phantom{\Big|\!\!} 0 & -\frac{1}{3}\partial_z w & \frac{1}{2}\partial_z v \\
\phantom{\Big|\!\!} \frac{1}{2}\partial_z u & \frac{1}{2}\partial_z v & \frac{2}{3}\partial_z w
\end{pmatrix}.$$
\end{minipage}
\par
Hence $2|\S| = \sqrt{(\partial_z u)^2 + (\partial_z v)^2 + \frac{4}{3}(\partial_z w)^2}$ so that the length $\ell$ (see figure) writes
$\ell^2 = 4|\S|^2 - \frac{1}{3}(\div \, \u)^2$. We then obtain
\begin{equation}
    \sin (\psi) = \frac{\div \, \u}{\ell} = \frac{\div \, \u}{\sqrt{4|\S|^2 - \frac{1}{3}(\div \, \u)^2}}.
\end{equation}
It is then possible to express the divergence of the velocity field as follows
\begin{equation}
    \div \, \u = \frac{2 |\S| \sin (\psi)}{\sqrt{1+\displaystyle \frac{1}{3}\sin^2(\psi)}}.
\end{equation}
When the dilatancy angle is small, the above relation is a close approximation to~\eqref{eq:85}. The coefficient $2$ remains valid in the three-dimensional case.
\end{remark}

If the dilatancy angle~$\psi$ remains small, we can use the following approximations\footnote{Curiously, it seems that the development used in~\cite[p.133]{andreotti12} (corrected in the next french edition, but repeated in~\cite{robinson23}) is not correct; the coefficient $\cos(\delta)$ being omitted. For their works, this does not change anything since this constant can be incorporated into the value of~$a$ hereafter.}
\begin{align}
    & \btau \approx \big( \sin (\delta) + \cos (\delta) \psi \big) \, p \frac{\S}{|\S|}\\
    & \div \, \u \approx 2|\S| \psi.
\end{align}
In terms of the notations previously introduced, these laws corresponds to the following functions 
\begin{align}
    & Z(\phi,I) = \sin (\delta) + \cos (\delta) \psi \label{Z1}\\
    & f(\phi,p,I) = \psi.\label{f1}
\end{align}
In a general case, the dilatancy angle~$\psi$ can depend on the inertial number~$I$ and the volume fraction~$\phi$. This angle must be zero at equilibrium, \textit{i.e.}   $\psi(\phi_{\mathrm{eq}}(I),I)=0$. The usual closure proposed by Roux and Radjaï~\cite{roux98} consists in using the first term of the expansion near the equilibrium state, namely
\begin{equation}\label{dilatance-RR}
    \psi(\phi,I) \approx a (\phi - \phi_{\mathrm{eq}}(I)) \qquad \text{where} \quad a=\frac{\partial \psi}{\partial \phi}(\phi_{\mathrm{eq}}(I),I).
\end{equation}
But it is clear that the choice~\eqref{dilatance-RR}, which implies the values of $Z$ via~\eqref{Z1} and~$f$ via~\eqref{f1} does not ensure that the stability conditions are fulfilled. In particular, the relation~\eqref{C1} linking~$Z$ and~$f$ is not satisfied.

The question is now to determine whether it is possible to close the system, \textit{i.e.} to propose an expression for the expansion angle~$\psi$, so that the stability conditions are satisfied.

Using~\eqref{Z1} and~\eqref{f1}, we can rewrite the condition~\eqref{C1} as
\begin{equation}\label{edp:psi}
    (2+\cos(\delta))I\partial_I\psi + 2(1-\cos(\delta)) \psi = 2 \sin(\delta).
\end{equation}
By specifying that at the equilibrium, we must have $\psi(\phi,I_{\mathrm{eq}}(\phi))=0$, we solve~\eqref{edp:psi} and we deduce that the dilatancy angle must have the following expression
\begin{equation}\label{psi}
    \psi(\phi,I) = \frac{\sin(\delta)}{1-\cos(\delta)} \Big( 1 - \Big( \frac{I_{\mathrm{eq}}(\phi)}{I}\Big)^\beta \Big) \qquad \text{where} \quad \beta = \frac{2(1-\cos(\delta))}{2+\cos(\delta)}.
\end{equation}
For common materials such as glass beads or sand, the angle of friction~$\delta$ is of the order of~$30^o$. The power~$\beta$ is then of the order of~$0.1$.

By using this dilatancy angle, the equations~\eqref{Z1} and~\eqref{f1} give an expression for~$Z$ and~$f$, and thus a complete model which fulfills, by construction, the stability condition~\eqref{C1}.
It is not difficult to verify that the other two stability conditions, namely~\eqref{C2} and~\eqref{C3}, are satisfied.

\begin{remark}
For dilatancy angle~$\psi$ given by~\eqref{psi}, it is therefore possible to provide a value for the "constant" introduced in the Roux and Radjaï's model~\cite{roux98}. Indeed, using the equation~\eqref{dilatance-RR}, we find
\begin{equation}
    a = \frac{2\sin(\delta)}{2 + \cos(\delta)} \frac{1}{\Delta\phi \, I}.
\end{equation}
\end{remark}

Finally, the fluidized and non-isochoric granular model based on the Drucker-Prager rheology incorporating the dilatancy effects writes
\begin{align}
& \partial_t \phi + \div ( \phi \, \u ) = 0, \label{syst1DP+} \\
& \partial_t ((1-\phi) p_f) + \div ( (1-\phi) p_f \u ) + p_{\mathrm{atm}} \div \, \u = p_{\mathrm{atm}} \div ( \kappa(\phi) \nabla p_f ), \label{syst2DP+} \\
& \phi \rho_s \big( \partial_t \u + \u \boldsymbol{\cdot} \nabla \u \big) = \phi \rho_s \g  - \nabla p + \div \Big( \big( \sin (\delta) + \cos (\delta) \psi(\phi,I) \big) p \frac{\S}{|\S|} \Big)  - \nabla p_f, \label{syst3DP+} \\
& \div \, \u = 2 |\S| \psi(\phi,I), \label{syst4DP+}
\end{align}
where $\psi(\phi,I)$ is given by~\eqref{psi}.

This model dissipates energy over time since the condition~\eqref{dissipation-general} (that is $Z\geq f$) is satisfied. Indeed, we have 
\begin{equation}
    Z-f = \sin(\delta) \Big( \frac{I_{\mathrm{eq}}(\phi)}{I}\Big)^\beta.
\end{equation}
so that the energy equation rewrites
\begin{equation}
    \frac{d \mathcal E}{dt} + \mathcal D = \int \phi \rho_s \g \cdot \u,
\end{equation}
where
\begin{equation}
\mathcal D = 2 \sin(\delta) \int \Big( \frac{I_{\mathrm{eq}}(\phi)}{I}\Big)^\beta p|\S| + \int \kappa(\phi) |\nabla p_f|^2.
\end{equation}

\subsubsection{The $\mu(I)$-rheology with dilatancy effect}

As with the Drucker-Prager model, the dilatancy effects can be taken into account in the $\mu(I)$-rheology. In~\cite{robinson23}, Robinson \textit{et al} proposed to modify the stress~$\btau$ and the divergence condition as follows
\begin{align}
    & \btau = \big( \mu(I) + \psi \big) p \frac{\S}{|\S|}\label{eq:86}\\
    & \div \, \u = 2|\S| \psi.\label{eq:87}
\end{align}
Using the same method as previously, we can determine the dilatancy angle~$\psi$ so that the model is stable. In that case, the condition~\eqref{C1} becomes
\begin{equation}
    \partial_I\psi = \frac{2}{3I}\mu(I) - \frac{1}{3}\mu'(I).
\end{equation}
By knowing the expression of the function~$\mu$ (see~\eqref{mu(I)}), it is easy to obtain the solution~$\psi$ which vanishes for $I=I_{\mathrm{eq}}(\phi)$, namely
\begin{equation}\label{psi:mu(I)}
\psi(\phi,I) = G(I)-G(I_{\mathrm{eq}}(\phi)),
\end{equation}
with
\begin{equation}
    G(I) = \frac{2\mu_1}{3}\ln(I/I_0) + \frac{\mu_2-\mu_1}{3} \Big( \frac{1}{1+I/I_0} + 2 \ln(1+I/I_0) \Big).
\end{equation}
The complete model is then very similar to the model~\eqref{syst1DP+}--\eqref{syst4DP+}, the dilatancy angle being given by the relation~\eqref{psi:mu(I)}. In this case, since $Z-f=\mu$, the energy dissipation is given by
\begin{equation}
\mathcal D = \int 2\mu(I) p|\S| + \int \kappa(\phi) |\nabla p_f|^2.
\end{equation}
Note also that in this case, the dilatancy law~\eqref{syst4DP+} can be approximated near the equilibrium by the Roux and Radjaï's model~\cite{roux98} $\div\,\u = 2a (\phi-\phi_{\mathrm{eq}}(I))|\S|$ where the "constant" $a$ is given by
\begin{equation}
    a = \frac{G'(I)}{\Delta\phi}.
\end{equation}


\section{Concluding remarks}
In this paper, we have studied non-isochoric fluidized granular models, \textit{i.e.}   which take into account the fluidisation effects of a compressible interstitial gas and local volume changes. The motivation is to model mixtures of particles, with a high concentration (between~$40$\% and~$60$\% in volume), and a gas, in particular air. Pyroclastic density flows, which are frequent and very devastating, fit into this framework. They are a major hazard in volcanic eruptions because of the great distances they can travel, up to 100 km in some cases, at high speed even on gentle slopes. 
The starting point of this study is the fluid-solid mixing model of Anderson and Jackson, which is simplified here because of the physical characteristics of the interstitial fluid, which is a gas. The compressibility of the latter allows us to transform the mass conservation equation of the gas phase into an equation for the pressure. The effect of the pressure of the interstitial gas is to reduce the friction between the particles and thus make it possible for the flow to accelerate and travel greater distances. This phenomenon has been observed in the laboratory (see~\cite{roche12})  showing that columns of particles (glass beads) travel twice as far when fluidised with an air stream injected from below. 
In order to close the equations for the solid phase, the constitutive laws, written as in Schaeffer~\textit{et al}~\cite{schaeffer19}, are specified in terms of a yield function and a dilatancy function, both of which depend on the volume fraction, inertial number and (solid) pressure. The resulting fluidized granular models is non-isochoric as it allows for volume change, namely the granular flow can expand when sheared. 
Moreover, we impose that the model be linearly stable, that it dissipate energy (over time), that it be compatible with the Roux and Radjaï~\cite{roux98} dilatancy model, and that the volume fraction, which is the solution of the mass conservation equation, be positive and bounded at all times.
Working with this theoretical framework, we have studied dilatancy laws that are compatible with classical rheologies,~\textit{i.e.} Drucker-Prager and~$\mu(I)$, with or without taking into account the effects of dilation in the yield condition. The main objective of this work was to derive fluidized and non-isochoric granular models that satisfy the aforementioned mathematical properties and take into account the known physics of granular materials, in particular in terms of rheology and deformation of the medium when sheared. 
The next steps will be to prove that these models are well-posed, \textit{i.e.} that they admit solutions that can be unique, and to propose stable numerical schemes. Then, numerical studies will have to be carried out to prove their validity, in particular by comparing the results of the simulations with laboratory experiments such as those described in~\cite{roche12}. These are the outlines of our future work on the subject.

\section*{Acknowledgments}
This is contribution no. XXX of the ClerVolc program of the International Research Center for Disaster Sciences and Sustainable Development of the University of Clermont Auvergne.
\medskip

Declaration of Interests. The authors report no conflict of interest.

\bibliography{biblio}
\bibliographystyle{abbrv}

\end{document}